\documentclass[11pt]{amsart}
\usepackage{amsthm}

\usepackage[all]{xy}
\usepackage{amssymb}
\usepackage{enumerate}
\usepackage{mathrsfs}
\usepackage{epsfig}
\usepackage{graphicx}
\usepackage{subfig}
\usepackage{float}
\usepackage{epigraph}
\usepackage{mathtools}

\numberwithin{equation}{section}

\newtheorem{thm}{Theorem}[section]

\newtheorem{lem}[thm]{Lemma}
\newtheorem{prop}[thm]{Proposition}
\newtheorem{rem}{Remark}[section]
\newtheorem{example}[thm]{Example}

\newtheorem{defin}[thm]{Definition}

\newcommand{\eq}[1]{(\ref{#1})}

\renewcommand{\Re}{\operatorname{\rm Re}}
\renewcommand{\Im}{\operatorname{\rm Im}}

\newcommand{\beqast}{\begin{eqnarray*}}
\newcommand{\eqast}{\end{eqnarray*}}
\newcommand{\beqa}{\begin{eqnarray}}
\newcommand{\eqa}{\end{eqnarray}}

\newcommand{\bbe}{\begin{equation}}
\newcommand{\ee}{\end{equation}}

\renewcommand{\Re}{\operatorname{\rm Re}}
\renewcommand{\Im}{\operatorname{\rm Im}}

\newcommand{\bC}{{\mathbb C}}
\newcommand{\bE}{{\mathbb E}}

\newcommand{\bQ}{{\mathbb Q}}

\newcommand{\bR}{{\mathbb R}}

\newcommand{\bZ}{{\mathbb Z}}

\newcommand{\cA}{{\mathcal A}}

\newcommand{\cF}{{\mathcal F}}

\newcommand{\cG}{{\mathcal G}}

\newcommand{\cS}{{\mathcal S}}

\newcommand{\cL}{{\mathcal L}}

\newcommand{\cC}{{\mathcal C}}

\newcommand{\cU}{{\mathcal U}}

\newcommand{\barX}{{\bar X}}
\newcommand{\uX}{{\underline X}}

\newcommand{\phipq}{{\phi^+_q}}
\newcommand{\phimq}{{\phi^-_q}}

\newcommand{\hG}{{\hat G}}

\newcommand{\Om}{{\Omega}}

\newcommand{\al}{\alpha}

\newcommand{\be}{\beta}

\newcommand{\de}{\delta}
\newcommand{\eps}{\epsilon}

\newcommand{\la}{\lambda}
\newcommand{\lp}{\lambda_+}
\newcommand{\lm}{\lambda_-}

\newcommand{\mum}{\mu_-}
\newcommand{\mup}{\mu_+}
\newcommand{\mumpr}{\mu'_-}
\newcommand{\muppr}{\mu'_+}

\newcommand{\num}{\nu_-}
\newcommand{\nup}{\nu_+}

\newcommand{\sg}{\sigma}

\newcommand{\om}{\omega}

\newcommand{\ze}{\zeta}

\newcommand{\ga}{\gamma}
\newcommand{\gap}{\gamma_+}
\newcommand{\gam}{\gamma_-}
\newcommand{\gappr}{\gamma'_+}
\newcommand{\gampr}{\gamma'_-}

\newcommand{\Ga}{\Gamma}

\newcommand{\dd}{\partial}

\newcommand{\bfo}{{\bf 1}}

\newcommand{\supp}{{\mathrm{supp}}}

\newcommand{\SM}{{\mathrm {SM}}}
\newcommand{\sSLM}{{\mathrm {sSLM}}}

\newcommand{\SLM}{{\mathrm {SLM}}}

\begin{document}

\title[L\'evy models amenable to efficient calculations]
{ L\'evy models amenable to efficient calculations}

\author[
Svetlana Boyarchenko and
Sergei Levendorski\u{i}]
{
Svetlana Boyarchenko and
Sergei Levendorski\u{i}}

\begin{abstract}
In our previous publications (IJTAF 2019, Math. Finance 2020), we
introduced a general class of SINH-regular processes and demonstrated
that efficient numerical methods for the evaluation of the
Wiener-Hopf factors and various probability distributions (prices of
options of several types) in L\'evy models can be developed using
only a few general properties of the characteristic exponent
$\psi$. Essentially all popular L\'evy processes enjoy these
properties. In the present paper, we define classes of
Stieltjes-L\'evy processes (SL-processes) as processes with
completely monotone L\'evy densities of positive and negative jumps,
and signed Stieltjes-L\'evy processes (sSL-processes) as processes
with densities representable as differences of completely monotone
densities. We demonstrate that 1) all crucial properties of $\psi$
are consequences of the representation
$\psi(\xi)=(a^+_2\xi^2-ia^+_1\xi)ST(\cG_+)(-i\xi)+(a^-_2\xi^2+ia^-_1\xi)ST(\cG_-)(i\xi)+(\sg^2/2)\xi^2-i\mu\xi$,
where $ST(\cG)$ is the Stieltjes transform of the (signed) Stieltjes
measure $\cG$ and $a^\pm_j\ge 0$; 2) essentially all popular
processes other than Merton's model and Meixner processes are
SL-processes; 3) Meixner processes are sSL-processes;
4) under a natural symmetry condition, essentially all popular
classes of L\'evy processes are SL- or sSL-subordinated Brownian
motion.

\end{abstract}

\thanks{
\emph{S.B.:} Department of Economics, The
University of Texas at Austin, 2225 Speedway Stop C3100, Austin,
TX 78712--0301, {\tt sboyarch@eco.utexas.edu} \\
\emph{S.L.:}
Calico Science Consulting. Austin, TX.
 Email address: {\tt
levendorskii@gmail.com}}

\maketitle



\tableofcontents

\section{Introduction}
The Fourier/Laplace transform and the Wiener-Hopf factorization
technique allow one to express probability distributions of L\'evy
processes, joint probability distributions of a L\'evy process and
its extrema, as well as prices of wide classes of options, as
integrals and repeated integrals.  Wiener-Hopf factors can also be
expressed as exponentials of integrals.  In many cases,
straightforward evaluation of these integrals using Fast Fourier
transform and/or Hilbert transform is inefficient, hence additional
tricks are needed.  Efficient calculations are possible if the
technique of conformal deformations is applied to each integral, and
then the simplified trapezoid rule is used.  The deformations in all
integrals must be in a certain agreement. See
\cite{iFT,paraHeston,paraLaplace,MarcoDiscBarr,one-sidedCDS,pitfalls,SINHregular,Contrarian,BSINH}
for details and comparison with other methods; note that in
\cite{Contrarian}, triple and quadruple integrals are efficiently
calculated.

The conformal deformation technique relies  on several  general properties of the characteristic exponent $\psi$; specific properties of 
different classes of processes are essentially irrelevant for efficient calculations.
The two most important properties are: (1) $\psi$ is analytic in a union of an open strip and a cone around or adjacent to $\bR$;  
(2) $\Re\psi(\xi)\to+\infty$ as $\xi\to \infty$ in an open sub-cone. In \cite{SINHregular},  we used properties (1) and (2)
to define a class of SINH-regular processes,  and mentioned that all popular classes of L\'evy processes are SINH-regular. 
If $X$ is SINH-regular, then flat contours of integration in the Fourier inversion formula and formulas for the Wiener-Hopf factors
can be deformed using  conformal maps of the form 
$\xi\mapsto i\om_1+b\sin(i\om+\xi)$, where a proper choice of $\om_1, \om\in \bR$ and $b>0$ depends on the properties of
the integrand.  The corresponding change of variables 
$\xi= i\om_1+b\sin(i\om+y)$ ({\em sinh-acceleration}) reduces
calculations to an integral over a horizontal line, with the
integrand decaying at infinity much faster than the initial one.
In the case of Laplace inversion, we use changes of variables of the
form $q=\sg+ib\sin(i\om+y)$, where $\sg, b, \om>0$.

The new integrand remains analytic in a strip around the line of
integration, hence the infinite trapezoid rule is very
efficient. The error of
the infinite trapezoid rule decays as $\exp[-2\pi d/\ze]$, where
$d$ is the half-width of the strip of analyticity around the line of
integration and $\ze$ is the step (see, e.g., Thm. 3.2.1 in
\cite{stenger-book}).  Therefore, if the strip of analyticity is not too
narrow, it is relatively easy to satisfy a very small error tolerance
for the discretization error.
Sinh-acceleration greatly increases the rate of decay at infinity,
and thus a moderate or even small number of terms in the simplified
trapezoid rule suffices to satisfy a small error tolerance
$\eps$. Typically, the complexity of the scheme is of the order of
$E\ln E$, where $E=\ln(1/\eps)$.

Let $q$ be the dual variable in the Laplace inversion formula. When
first passage probabilities are calculated and barrier and lookback
options are priced, the Wiener-Hopf factors must be calculated for
all points $q$ on a deformed contour in the Bromwich integral, hence
zeros of $q+\psi(\xi)$ must be avoided or, if some of the zeros are
crossed, these zeros must be accurately calculated and the residue
theorem needs to be applied. Therefore, it is necessary to know where
these zeros are located.  If the Gaver-Stehfest method is used, then
only $q>0$ are used. Almost all popular L\'evy models have two
additional properties: (3) the cone of analyticity of $\psi$ is
$\bC\setminus i\bR$, and (4) for $q>0$, equation $q+\psi(\xi)=0$ has
no solutions in $\bC\setminus i\bR$. Hence, the problem of locating
the zeros becomes trivial.  Conditions (1)-(4) are used in
\cite{paired} to define a class of strongly regular L\'evy processes
of exponential type.

In the case of stable L\'evy processes, there is no strip of
analyticity of $\psi$ (formally, the strip degenerates into $\bR$),
but $\psi$ enjoys properties (2)-(4) in the cone $\bC\setminus i\bR$,
and efficient modifications of the procedures developed for
SINH-processes are possible. See \cite{ConfAccelerationStable}, where
efficient procedures for the evaluation of stable distributions are
developed and the relative efficiency of several families of
conformal deformations is discussed.

Popular classes of L\'evy processes are constructed either by defining
the L\'evy measure directly or via subordination of the Brownian
motion.  In the present paper, we suggest a general construction of a
wide class of processes, which enjoy properties (1)-(4) and contain
all popular models except for Merton's model.
 We construct the Laplace exponents of the positive and negative jumps generalizing the definition of the Bernstein functions,
  and  demonstrate that the crucial properties of $\psi$
are consequences of the representation 
\bbe\label{eq:sSLrepr}
\psi(\xi)=(a^+_2\xi^2-ia^+_1\xi)ST(\cG^0_+)(-i\xi)+(a^-_2\xi^2+ia^-_1\xi)ST(\cG^0_-)(i\xi)+(\sg^2/2)\xi^2-i\mu\xi, 
\ee
where $ST(\cG)$ is the Stieltjes transform of the (signed) Stieltjes measure $\cG$,  $a^\pm_j\ge 0$, and $\sg^2\ge0$, $\mu\in\bR$.
 If $\cG^0_\pm\ge 0$, we call $X$ a Stieltjes-L\'evy process
(SL-process\footnote{Note that the Laplace exponent of an
SL-subordinator is a complete Bernstein function.
}). SL-processes enjoy properties (1)-(4), and the L\'evy densities
of positive and negative jump components are completely monotone. If
at least one of $\cG_\pm$ is a signed measure, then (1)-(3) continue
to hold but (4) may fail; the L\'evy density of the corresponding
jump component is the difference of two completely monotone
functions, thus it can be non-monotone. We say that $X$ is a signed
Stieltjes-L\'evy process (sSL-process).

The rest of the paper is organized as follows. In Section
\ref{s:SINHregular}, we give the definition of SINH-regular
processes. We slightly change the definition given in
\cite{SINHregular}, calculate the strips and cones of analyticity for
essentially all popular classes of L\'evy processes, and derive upper
and lower asymptotic bounds for $|\psi|$ and $\Re \psi$,
respectively. These bounds are necessary ingredients for the
conformal acceleration method.  In Section \ref{s:ST_and_sSSL}, we
derive a representation of the L\'evy density in terms of the
characteristic exponent, which naturally leads to the definitions of
sSL- and SL-processes. We establish several useful properties of
these processes, derive a representation of the measures associated
with sSL- and SL-processes in terms of the characteristic exponent,
and determine sufficient conditions for sSL- and SL-processes to be
SINH-regular. These conditions can be used to find good
approximations of KoBoL \cite{KoBoL} and other processes defined by
absolutely continuous measures with
hyperexponential L\'evy processes and more general meromorphic processes 
defined by discrete measures. SL-processes defined by discrete
measures with accumulation points at $0$ and $+\infty$ can be used to
approximate stable L\'evy processes.  Non-existence of solutions of
the equation $q+\psi(\xi)$ on $\bC\setminus i\bR$ (when $q>0$) for
SL-processes is proved in Section \ref{s:sol_SL}.  In Section
\ref{s:mixing_sub}, we prove that, under additional conditions,
(i) mixtures of SINH-regular, sSL- and SL-processes are SINH-regular,
sSL- and SL-processes, respectively; (ii) SINH-regular processes
subordinated by SINH-processes are SINH-regular; and (iii) SL-processes
subordinated by sSL-subordinator (resp., by SL-subordinator) are
sSL-processes (resp., SL-processes).  As a byproduct, we prove that
if, for some $\be\in\bR$, $\psi(\xi-i\be)=\psi(i\be-\xi),
\forall\ \xi\in \bR$, an SL-process is a Brownian motion (BM)
subordinated by an SL-process, and derive explicit formulas for the
latter (in the case of sSL-processes, an additional condition is
needed). Thus, under the symmetry condition above, essentially all
popular classes of L\'evy processes are SL- or sSL-subordinated BM.
This result unifies and generalizes the well-known constructions of
the Variance Gamma processes (VGP) \cite{MM91}, Normal Inverse
Gaussian (NIG) processes \cite{B-N} and Normal Tempered Stable (NTS)
processes \cite{B-N-L} by subordination, and representations of a
symmetric KoBoL (CGMY) and Meixner L\'evy processes as subordinated
BM \cite{madan-yor}.  In Section \ref{concl}, we summarize the
results of the paper and outline possible extensions.  The technical
details are relegated to Section \ref{tech}. For the reader's
convenience, in Section \ref{s:CONF}, we outline applications of the
conformal acceleration method to several basic situations.

\section{SINH-regular  L\'evy processes}\label{s:SINHregular}
\subsection{Definitions}\label{main_def}
Let $X$ be the L\'evy process on $\bR$, let $(\Om; \cF; \{\cF_t\}_{t\ge 0})$ be the filtered probability space 
  generated by $X$, and let $\bQ$ be a probability measure on $(\Om; \cF; \{\cF_t\}_{t\ge 0})$.
 $\bE=\bE^\bQ$ denotes the expectation operator under $\bQ$.
 We use the definition in \cite{KoBoL,NG-MBS} of the characteristic exponent $\psi(\xi)=\psi^\bQ(\xi)$ of a L\'evy process $X$ 
under  $\bQ$, which is marginally different from the definition in \cite{sato}. Namely, $\psi$ is definable from
$
\bE[e^{i\xi X_t}]= e^{-t\psi(\xi)}$. The L\'evy-Khintchine formula is
\bbe\label{eq:LevyKhintchine1D}
\psi(\xi)=\frac{\sg^2}{2}\xi^2-i\mu\xi+\int_{\bR\setminus 0} (1-e^{ix\xi}+\bfo_{(-1,1)}(x)ix\xi)F(dx),
\ee
where $\sg^2\ge 0, \mu\in\bR$ and the measure $F(dx)$ satisfies $\int (x^2\wedge1) F(dx)<\infty$.  $X$ is of finite variation
iff $\int (|x|\wedge1) F(dx)<\infty$, and then the term $\bfo_{(-1,1)}(x)ix\xi$ can (and will) be omitted. 

\vskip0.1cm
\noindent 
{\sc Notation.} Throughout the paper, $\mu_\pm, \ga_\pm, \ga'_\pm, \ga$ are reals satisfying
 $\mum\le 0\le \mup$,  $\mum<\mup$,   $-\pi/2\le \gam\le 0\le \gap\le \pi/2$, $\gam<\gampr<\gappr<\gap$, and $\ga\in (0,\pi]$. 
  We define   cones
 $\cC_{\gam,\gap}=\{e^{i\varphi}\rho\ |\ \rho> 0, \varphi\in (\gam,\gap)\cup (\pi-\gap,\pi-\gam)\}$, 
 $\cC_{\ga}=\{e^{i\varphi}\rho\ |\ \rho> 0, \varphi\in (-\ga,\ga)\}$, and the strip $S_{(\mum,\mup)}=\{\xi\ |\ \Im\xi\in (\mum,\mup)\}$.

As in  \cite{iFT,paired,SINHregular}, we represent the characteristic exponent in the form
\begin{equation}\label{reprpsi}
\psi(\xi)=-i\mu\xi+\psi^0(\xi),
\end{equation}
and impose  conditions on $\psi^0$.

 \begin{defin}\label{def:SINH_reg_proc_1D0}
(1) We say that $X$ is a SINH-regular L\'evy process  (on $\bR$) of type $((\mum,\mup);\cC; \cC_+)$ and order
 $\nu\in (0,2]$,
 iff
the following conditions are satisfied:
\begin{enumerate}[(i)]
\item
$\mum<0\le \mup$ or $\mum\le 0<\mup$;
\item
$\cC=\cC_{\gam,\gap}, \cC_+=\cC_{\gampr,\gappr}$, where $\gam<0<\gap$, $\gam\le \gampr\le 0\le \gappr\le \gap$,
and $|\gampr|+\gappr>0$;  
\item
$\psi^0$ admits analytic continuation to $i(\mum,\mup)+ (\cC\cup\{0\})$;

\item
for any $\varphi\in (\gam,\gap)$, there exists $c_\infty(\varphi)\in \bC\setminus (-\infty,0]$ s.t.
\begin{equation}\label{asympsisRLPE}
\psi^0(\rho e^{i\varphi})\sim  c_\infty(\varphi)\rho^\nu, \quad \rho\to+\infty;
\end{equation}
\item
the function $(\gam,\gap)\ni \varphi\mapsto c_\infty(\varphi)\in \bC$ is continuous;
\item
for any $\varphi\in (\gampr, \gappr)$, $\Re c_\infty(\varphi)>0$.
\end{enumerate}

(2) We say that $X$ is a SINH-regular L\'evy process  (on $\bR$) of type $((\mum,\mup);\cC; \cC_+)$ and order $\nu=1+$,
 iff  
the conditions above bar \eq{asympsisRLPE} are satisfied, and \eq{asympsisRLPE} is replaced with
\begin{equation}\label{asympsisRLPE1p}
\psi^0(\rho e^{i\varphi})\sim c_\infty(\varphi)\rho\ln\rho, \quad \rho\to+\infty.
 \end{equation}
(3) We say that $X$ is a SINH-regular L\'evy process  (on $\bR$) of type $((\mum,\mup);\cC; \cC_+)$ and order $\nu=0+$,
 iff  
the conditions (i)-(iii) are satisfied, and, as $\xi\to\infty$ remaining in $i(\mum,\mup)+\cC$, 
\begin{equation}\label{asympsisRLPE0p}
\psi^0(\xi)\sim c\ln|\xi|, \end{equation}
where $c>0$.
\end{defin}

\begin{rem}\label{rem:ell_def}{\rm \begin{enumerate}[(1)]
\item
Conditions for $\varphi\in (\pi-\gam,\pi-\gap)$ and $\varphi\in (\pi-\gampr,\pi-\gappr)$
follow from the conditions for $\varphi\in (\gam,\gap)$ and $\varphi\in (\gampr,\gappr)$ because
$\psi(-\bar\xi)=\overline{\psi(\xi)}$.
\item In Definition \ref{def:SINH_reg_proc_1D0},  conditions are imposed on $\psi^0$ whereas in \cite{SINHregular},
the same conditions are imposed on $\psi$. 
If either $\nu\in (1,2]$ or $\mu=0$, then the conditions on $\psi^0$ and $\psi$ are equivalent,
and the process is  an elliptic\footnote{The name elliptic is natural
from the point of view of the theory of PDO: if \eq{asympsisRLPE} holds, then the infinitesimal generator
$L^0=-\psi^0(D)$ is an elliptic PDO.} SINH-process of order $\nu$
in the terminology of \cite{SINHregular}.
\item
If $\nu<1$ and $\mu\neq 1$, then the process $X$ with the characteristic exponent $-i\mu\xi+\psi^0(\xi)$ is elliptic of 
order 1, in the terminology of \cite{SINHregular}. Furthermore, according to the definition in  \cite{SINHregular},
 $\cC_+$ is determined by the drift term if $\mu\neq 0$: $\cC_+$ is the intersection of $\cC$ with
the upper (resp., lower) half-plane if $\mu>0$ (resp., $\mu<0$).
 
\end{enumerate}
}\end{rem}

\begin{defin}\label{def:ell_def_ordering}
For $\nu=0+$ and $\rho>1$, set $\rho^{\nu}=\ln\rho$. For $\nu=1+$ and $\rho>1$, set $\rho^\nu=\rho\ln\rho$.  
 A linear ordering in the set $(0,2]\cup\{0+\}\cup\{1+\}$ is defined as follows: 1) for $\nu_1, \nu_2\in (0,2]$, the usual ordering;
2) $0+<\nu$ for any $\nu\in (0,2]\cup\{1+\}$; 3) $1<1+$; 4) $1+<\nu$ for any $\nu\in (1,2]$.
\end{defin}
 We see that we can write \eq{asympsisRLPE1p} and \eq{asympsisRLPE0p}
 in the form \eq{asympsisRLPE} with $\nu=1+$ and $\nu=0+$, respectively, and use \eq{asympsisRLPE} for all $\nu\in (0,2]\cup\{0+\}\cup\{1+\}$.


In \cite{SINHregular}, we introduced a more general definition of SINH-regular processes, in terms of the upper and lower weight functions
$w_{up}$ and $w_{low}$; the former was used in the upper bound for $|\psi(\xi)|$ on $i(\mum,\mup)+(\cC\cup\{0\})$, the latter
in the lower bound for $\Re\psi(\xi)$ on $i(\mumpr,\muppr)+(\cC_+\cup\{0\})$, where $\mum\le \mumpr\le0\le\muppr\le\mup$, $\mumpr<\muppr$. 


In the paper, for the sake of brevity, we simplify the system of definitions in \cite{SINHregular}, and use a pair $(\nu',\nu)$, $\nu',\nu\in (0,2]\cup\{0+,1+\}$
to define the order if \eq{asympsisRLPE} does not hold. We allow $\nu'=\nu$. In basic examples that we
consider and construct, $\nu'=\nu$. Examples with $\nu'<\nu$ can be constructed mixing processes of different orders.
See also Remark \ref{rem:1p1}. 

 \begin{defin}\label{def:SINH_reg_proc_1D}
We say that $X$ is a SINH-regular L\'evy process  (on $\bR$) of type $((\mum,\mup);\cC; \cC_+)$ and order
 $(\nu',\nu)$ (lower order $\nu'$ and upper order $\nu$), where $\nu',\nu\in (0,2]\cup \{0+,1+\}$,
 iff conditions (i)-(iii) of Definition \ref{def:SINH_reg_proc_1D}
are satisfied, and 
\begin{enumerate}[(i)]
\item
$\forall$\ $\ga_{1,-}\in (\gam,0)$ and $\ga_{1,+}\in(0,\gap)$, $\exists$\ $C,R>0$ s.t. $\forall$\  $\varphi\in [\ga_{1,-},\ga_{1,+}]$
and  $\rho>R$,
\bbe\label{upperbound_gen}
|\psi(\rho e^{i\varphi})|\le C\rho^\nu;
\ee
\item
$\forall$\ $\ga'_{1,-}\in (\gampr,\gappr)$ and $\ga_{1,+}\in(\ga'_{1,-},\gappr)$, $\exists$\ $c,R>0$ s.t. $\forall$\  $\varphi\in [\ga'_{1,-},\ga'_{1,+}]$
and  $\rho>R$,
\bbe\label{lowerbound_gen}
\Re \psi(\rho e^{i\varphi})\ge c\rho^{\nu'}.
\ee 
\end{enumerate}
\end{defin}

\begin{rem}\label{rem:1p1}{\rm If $\bR\setminus \{0\}\not\subset \cC_+$, hence, $\bC_+$ is adjacent to the real line but does not contain  $\bR\setminus \{0\}$, then, in order to ensure the convergence of
integrals in pricing formulas and justify conformal deformations of the contours of integration, we need to use lower bounds
for $\Re\psi^0(\xi)$ on $\bR\cup\cC_+$. In all examples  constructed in the paper, if $\cC_+$ is adjacent to $\bR$ but does not contain 
$\bR\setminus \{0\}$, $X$ is of  order $1+$ and the lower bound on $\cC_+\cup (\bR\setminus \{0\})$ is valid with $|\xi|$ instead of
$|\xi|\ln (2+|\xi|)$. The process is of order $(1,1+)$. }
\end{rem}

\subsection{Spectrally positive and spectrally negative SINH-processes}\label{def_SINH-one-sided}
Domains of analyticity of spectrally one-sided  SINH-processes are wider than the ones in Definition
\ref{def:SINH_reg_proc_1D0}:
\begin{enumerate}[i)]
\item
if there are no negative  jumps,  the strip of analyticity  is of the form $S_{(\mum,+\infty)}$, where $\mum\le 0$,
and  $\cC=i\cC_{\ga}$, where $\ga\in (\pi/2,\pi)$. 
\item
if there are no positive jumps,  the strip of analyticity  is of the form $S_{(-\infty,\mup)}$, where $\mup\ge 0$,
and $\cC=-i\cC_{\ga}$, where $\ga\in (\pi/2,\pi)$. 
\end{enumerate}

\begin{prop}\label{sp_pos_neg}
\begin{enumerate}[(a)]
\item
A spectrally positive SINH-process of the upper order $\nu\in \{0+\}\cup(0,1)$, with  non-negative drift, is a subordinator.
\item
A spectrally negative SINH-process of the upper order $\nu\in \{0+\}\cup(0,1)$, with  non-positive drift, is the dual process to a subordinator.
\end{enumerate}
\end{prop}
\begin{proof} (a) Let $x<0$ and $\mu\ge 0$. Then, in the formula for the pdf
\bbe\label{pdf}
p_t(x)=\frac{1}{2\pi}\int_{\Im\xi=\om} e^{-ix\xi-t(-i\mu\xi+\psi^0(\xi))}d\xi,
\ee
where $\om>\mum$, we can push the line of integration up: $\om\to+\infty$, and, in the limit, obtain $p_t(x)=0$. 
(b) is immediate from (a). 
\end{proof}

\begin{rem}\label{one-sided stable}{\rm 
One-sided stable processes are SINH-regular since $\mum=0$ or $\mup=0$ are allowed. For general stable L\'evy
processes,  the strip of analyticity does not exist. Formally, $\mum=\mup=0$.

}
\end{rem}

\subsection{Examples}\label{ss:examplesSINH}
Essentially all L\'evy processes
used in quantitative finance are  SINH-regular.
 \subsubsection{The Brownian motion (BM)} BM is of  order 2; since $\psi^0(\xi)=\frac{\sg^2}{2}\xi^2$ is an entire function,  $\cC=\bC$, $\mum=-\infty, \mup=+\infty$. For any $\varphi\in [0,2\pi)$, 
 \bbe\label{asBM}
 \psi^0(\rho e^{i\varphi})\sim \frac{\sg^2}{2}\rho^2 e^{2i\varphi},\ \rho\to+\infty,
 \ee
 hence, $c_\infty(\varphi)= \frac{\sg^2}{2}e^{2i\varphi}$ has a positive real part iff $\cos(2\varphi)>0$. It follows that 
 $\cC_+=\cC_{-\pi/4,\pi/4}$.
 
 \subsubsection{Merton model \cite{merton-model}}\label{ss:merton-model} The characteristic exponent is given by
 \bbe\label{Merton_psi}
 \psi^0(\xi)=\frac{\sg^2}{2}\xi^2+\la\cdot\bigl(1-e^{im\xi-\frac{s^2}{2}\xi^2}\bigr),
 \ee
where $\sg,s,\la>0$ and $\mu,m\in\bR$. As far as the analytical properties formulated in the definition of SINH-processes are concerned,
the difference with BM is that $\cC=\cC_+=\cC_{-\pi/4,\pi/4}$, and $\cC=\cC_{\gam,\gap}$ with either $\gam<-\pi/4$ or $\gap>\pi/4$ cannot be used.

 \subsubsection{L\'evy processes with rational characteristic exponents 
and non-trivial BM component}\label{ss:Levy_rat}  The order is 2, and an admissible  strip of analyticity 
 $S_{(\mum,\mup)}$ around the real axis may not
contain poles of $\psi^0$.  Explicit formulas for the Wiener-Hopf factors are easy to derive (see, e.g.,   \cite{NG-MBS}) in terms of poles of $\psi^0$, zeros of the function $q+\psi(\xi)$, the multiplicities of zeros and poles being taken into account. After $\mum,\mup$ are chosen, $\cC$ is the maximal cone around the real axis such that
$\cU=i(\mum,\mup)+(\cC\cup\{0\})$ contains no poles, and $\cC_+=\cC\cap \cC_{-\pi/4,\pi/4}$. Since the efficiency of SINH-acceleration depends, mostly, on the ``width" of $\cC_+$, it is advisable to choose small (in absolute value) $\mum$ and $\mup$ so that $\cC_+$ can be chosen ``wider".
L\'evy processes of the phase type \cite{AsmussenRuin,AAP} have rational characteristic exponents, hence, the recommendations above are applicable. 

Calculation of the rational characteristic exponent is straightforward if the L\'evy densities of positive and negative jumps are
mixtures of exponential polynomials. Furthermore, all poles are on $i\bR$, hence,
$\cC=\bC\setminus i\bR$  \cite{amer-put-levy-maphysto, amer-put-levy}. The factorization of  $q+\psi(\xi)$ (calculation of the Wiener-Hopf factors) simplifies if all the roots of the characteristic equation $\psi(\xi)+q=0$ are on the imaginary axis. Then the roots can be easily calculated,
and explicit formulas for the Wiener-Hopf factors as sums or products derived. See, e.g., \cite{amer-put-levy-maphysto, amer-put-levy}.
A popular special case is
the hyper-exponential jump-diffusion model (HEJD model) introduced in 
   \cite{amer-put-levy-maphysto} without a special name assigned and  \cite{lipton-risk},
   and studied in detail in 
   \cite{amer-put-levy-maphysto, amer-put-levy}). The L\'evy measure is  of the form
   \bbe\label{eq:measure_HEJD}
F(dx)=\bfo_{(-\infty,0)}(x)\sum _{k=1}^{n^-} p^-_j \al^-_j e^{\al^-_j x}+
\bfo_{(0,+\infty)}(x)\sum _{j=1}^{n^+} p^+_j \al^+_j e^{-\al^+_j x},
\ee
where $n^\pm$ are positive integers,  and
$\al_j^\pm, p_j^\pm>0$ are reals. The characteristic
exponent is  
\begin{equation}\label{e:char-exp-HEJD}
\psi^0(\xi) = \frac{\sg^2}{2}\xi^2-i\mu\xi+\sum_{j=1}^{n^+}
p_j^+\frac{-i \xi}{\al_j^+-i\xi} + \sum_{k=1}^{n^-}p_k^-
\frac{i  \xi}{\al_k^-+i\xi},
\end{equation}
Double-exponential jump diffusion model introduced to finance in \cite{Kou} (and well-known for decades)
is a special case of hyper-exponential
jump-diffusion models  with $n^+=n^-=1$. The order is 2, $\mup=\min \al^-_k$, $\mum=-\min \al^+_k$,
and $\cC$, $\cC_+$ are as in the BM model. 

In \cite{CaiKou11}, a class of processes with the L\'evy measure of the form \eq{eq:measure_HEJD} with some of $p^\pm_j$
being negative is introduced, and the name mixed exponential jump diffusion model (MEJD) is suggested.
Sufficient conditions for $p^\pm_j$ and $\al^\pm_j$ to define the non-negativity of the densities are 
$p^\pm_1 > 0$ and $\sum_{j=1}^k p^\pm_j \al^\pm_j\ge 0$, $k=1,2,\ldots, n^\pm$. An important qualitative difference
between HEJD and MEJD is that in HEJD models, the L\'evy densities of positive and negative jumps are monotone
(in fact, completely monotone), whereas in MEJD, the densities may be non-monotone. Note that the  L\'evy densities given by mixtures of exponential polynomials  \cite{amer-put-levy-maphysto, amer-put-levy} are typically non-monotone, and qualitative properties of MEJD densities can be easily reproduced by exponential polynomials. As the simplest example, the reader can compare the following two functions on $\bR_+$:
$f_1(x)=e^{-\la_1 x}-e^{-\la_2 x}$, where $0<\la_1<\la_2$, and $f_2(x)=x e^{-\la_1 x}$.

\subsubsection{Variance Gamma processes (VGP)} VG model was introduced to Finance in \cite{MM91}. The characteristic exponent can be written in the form
\begin{equation}\label{VGPexp}
\psi^0(\xi)=c[\ln(\al^2-(\be+i\xi)^2)-\ln (\al^2-\be^2)],
\end{equation}
where $\al>|\be|\ge 0$, $c>0$. VGP is SINH-regular of  type $((-\al+\be,\al+\be); \bC\setminus i\bR, \bC\setminus i\bR)$ 
and order $0+$ because $\forall\ \varphi\in (-\pi/2,\pi/2)$, 
\bbe\label{asVGP}
\psi^0(\rho e^{i\varphi})= c(\ln\rho+i\varphi)+O(1), \ \rho\to+\infty.
\ee

 \subsubsection{NIG and NTS}
 Normal inverse Gaussian (NIG) processes, and the generalization: Normal Tempered Stable (NTS) processes are constructed
 in \cite{B-N,B-N-L}, respectively. The characteristic exponent is given by
 \begin{equation}\label{NTS2}
\psi^0(\xi)=\de[(\al^2-(\be+i\xi)^2)^{\nu/2}-(\al^2-\be^2)^{\nu/2}],
\end{equation}
where $\nu\in (0,2)$, $\de>0$, $|\be|<\al$; NIG obtains with $\nu=1$.  This is a SINH-regular process  of order $\nu$ and type $((-\al+\be,\al+\be); \bC\setminus i\bR, \cC_{-\ga_\nu,\ga_\nu})$,
where 
$\ga_\nu=\min\{1,1/\nu\}\pi/2$. Indeed, for $\varphi\in (-\pi/2,\pi/2)$,
\bbe\label{asNTS}
\psi^0(\rho e^{i\varphi})= \de e^{i\varphi\nu}\rho^\nu+O(\rho^{\nu-1})+O(1), \ \rho\to+\infty.
\ee

  \subsubsection{The Meixner process}\label{sss:Meixner}
 For the background, see, e.g., \cite{SchoutensTeugels1998, PitmanYor2003,madan-yor}.  The L\'evy density of the Meixner process $X$ is
 \bbe\label{eq:Meixner_dens}
 f(x)=\de\frac{\exp(b x/a)}{x\sinh(\pi x/a)},
 \ee
 where $\de, a>0$ and the asymmetry parameter $b\in(-\pi,\pi)$. The characteristic exponent is 
 \bbe\label{eq:Meixner_ch_exp}
 \psi^0(\xi)=2\de[\ln[\cosh((a\xi-ib)/2))]-\ln \cos(b/2)].
 \ee 
 The formula $\ln\cosh(z)=z+\ln(1+e^{-2z})-\ln 2$ defines a function analytic in the right half-plane, hence, $\psi^0(\xi)$ admits analytic continuation
 to $\bC\setminus i\bR$. Set $\mum=(-\pi+b)/a,
 \mup=(\pi+b)/a$. Since $\cosh((a\xi-ib)/2))>0$ for $\xi\in i(\mum,\mup)$, $\psi^0$ is analytic in $\bC\setminus i((-\infty,\mum]\cup [\mup,+\infty))$.
Let $\ga\in (0,\pi/2)$.  As $\xi\to \infty$ in $\cC_\ga$, 
$
\psi(\xi)\sim a\de \xi$, therefore, $X$ is SINH-regular of order 1 and type $(((-\pi+b)/a,(\pi+b)/a), \bC\setminus i\bR, \bC\setminus i\bR)$.

 \subsubsection{KoBoL processes}
 A generic process of Koponen's family  \cite{KoBoL} was constructed as a mixture of spectrally negative and positive pure jump processes, with the L\'evy measure
\begin{equation}\label{KBLmeqdifnu}
F(dx)=c_+e^{\lm x}x^{-\nu_+-1}\bfo_{(0,+\infty)}(x)dx+
 c_-e^{\lp x}|x|^{-\nu_--1}\bfo_{(-\infty,0)}(x)dx,
\end{equation}
where $c_\pm>0, \nu_\pm\in [0,2), \lm<0<\lp$. In this paper, we allow for $c_+=0$ or $c_-=0$,
 $\lm=0<\lp$ and $\lm<0\le \lp$. This generalization is almost immaterial for evaluation of probability distributions and expectations because
 for efficient calculations, the first crucial property, namely, the existence of a strip of analyticity of the characteristic exponent, around or adjacent to the real line, holds if $\lm<\lp$ and $\lm\le 0\le \lp$. \footnote{The property does not hold
 if there is no such a strip (formally, $\lm=0=\lp$). The classical example are stable L\'evy processes.
 The conformal deformation technique  can be modified for this case as well \cite{ConfAccelerationStable}.}
 Furthermore, the Esscher transform allows one to reduce both cases $\lm=0<\lp$ and $\lm<0\le \lp$ to the case $\lm<0<\lp$. 
 Using the L\'evy-Khintchine formula, it is straightforward to
derive from \eq{KBLmeqdifnu}   
explicit formulas for $\psi^0$. See \cite{KoBoL,NG-MBS}.
If $\nu_\pm\in (0,2), \nu_\pm\neq 1$,
\bbe\label{KBLnupnumneq01}
\psi^0(\xi)=c_+\Ga(-\nu_+)((-\lm)^{\nu_+}-(-\lm-i\xi)^{\nu_+})+c_-\Ga(-\nu_-)(\lp^{\nu_-}-(\lp+i\xi)^{\nu_-}).
\ee
(Formulas in the case $\nu_\pm=0,1$ are in Sect. \ref{OnesidedKoBoL1}
and \ref{OnesidedKoBoL0}.)
A subclass with $\nu_+=\nu_-=\nu\in (0,2)$ and $c_+=c_-$
 was labelled  KoBoL in \cite{NG-MBS} and called a process of order $\nu$.   To simplify the name, we will call a pure jump process with the L\'evy measure \eq{KBLmeqdifnu}
 a KoBoL process as well. As in \cite{KoBoL}, we allow for $\nu_-\neq \nu_+$ and $c_+\neq c_-$.
  If either $c_-=0$ or $c_+=0$, we say that the process is a one-sided KoBoL. The formula \eq {KBL1p} for one-sided KoBoL of order 1
   is different.
 One-sided KoBoL processes were used in \cite{one-sidedCDS} to price CDSs. Mixing one-sided processes of order $\nu\in (0,2),\nu\neq 1,$ and order 1, one can obtain more exotic characteristic exponents.
 
 Note that a specialization
 $\nu\neq 1$, $c=c_\pm>0$, of KoBoL was named CGMY model in \cite{CGMY} (and the labels were changed:
 letters $C,G,M,Y$ replace the parameters $c,\nu,\lm,\lp$ of KoBoL):
 \bbe\label{KBLnuneq01}
 \psi^0(\xi)= c\Ga(-\nu)[(-\lm)^{\nu}-(-\lm- i\xi)^\nu+\lp^\nu-(\lp+ i\xi)^\nu].
\ee
Evidently, $\psi^0$ given by \eq{KBLnuneq01} is analytic in $\bC\setminus i\bR$, and  $\forall\ \varphi\in (-\pi/2,\pi/2)$, \eq{asympsisRLPE} holds with
\bbe\label{ascofnupeqnumcc}
c_\infty(\varphi)=-2c\Ga(-\nu)\cos(\nu\pi/2)e^{i\nu\varphi}.
\ee
Hence, $X$ is SINH-regular of  type $((\lm,\lp), \bC\setminus \{0\}, \cC_{-\ga_\nu,\ga_\nu})$, 
where $\ga_\nu=\min\{1,1/\nu\}\pi/2$, and order $\nu$. For the calculation of order and type for a generic KoBoL,
see Section \ref{ss:KoBoL}.

 \subsubsection{The $\be$-class \cite{beta}}\label{ss:beta}  The characteristic exponent is of the form
      \beqa\label{BetaChExp}
      \psi^0(\xi)&=&\frac{\sg^2}{2}\xi^2+\frac{c_1}{\be_1}\left\{B(\al_1,1-\ga_1)-B(\al_1-\frac{i\xi}{\be_1},1-\ga_1)\right\}
      \\\nonumber  &&
    +\frac{c_2}{\be_2}\left\{B(\al_2,1-\ga_2)-B(\al_2+\frac{i\xi}{\be_2},1-\ga_2)\right\},
      \eqa
      where $c_j\ge 0, \al_j, \be_j>0$ and $\ga_j\in (0,3)\setminus\{1,2\}$, and $B(x,y)$ is the Beta-function.
      It was shown in \cite{beta}, that all poles of $\psi^0$ are on $i\bR_+\setminus 0$. Hence, $(\mum,\mup)$ 
      is the maximal interval containing 0 and no poles of $\psi^0$, and $\cC=\bC\setminus i\bR$. For calculation of the order
      of the process and $\cC_+$ as functions of the parameters in \eq{BetaChExp}, see Section \ref{Calculations in the beta-model}.
      The number of parameters is larger than in the case of KoBoL but
      the variety of different cases (order and type) is essentially the same as in the case of KoBoL. 
    
\subsubsection{
Meromorphic L\'evy processes \cite{KuzKyprPardo}}\label{ss:meromorphic} The L\'evy measures and  characteristic exponents of the meromorphic L\'evy processes
      introduced in \cite{KuzKyprPardo} are defined by almost the same formulas 
      as in HEJD model. The difference is that the sums in \eq{eq:measure_HEJD} and \eq{e:char-exp-HEJD}
      are infinite. A natural condition $\al^\pm_j\to +\infty$ as $j\to +\infty$
      is imposed, and the requirement that the infinite sum defines a L\'evy measure is equivalent to 
      $\sum_{j\ge 0}p^\pm_j(\al^\pm_j)^{-2}<+\infty.$ 
      The poles of $\psi(\xi)$ are on $\bR\setminus 0$.  If $\sg^2>0$, 
      meromorphic processes are  SINH-regular  of order 2 and type $((\mum\mup), \bC\setminus i\bR, \cC_{-\pi/4,\pi/4})$,
      where $(\mum,\mup)$ is the maximal interval containing 0 and no poles of $\psi^0$.
      If $\sg^2=0$, additional conditions on the asymptotics of $p^\pm_j$ and $\al^\pm_j$ as $j\to+\infty$ need to be imposed  to obtain a SINH-regular process.
      See Section \ref{s:regST_and_sSSL}.
    
   \subsection{SINH-regular infinitely divisible distributions and related L\'evy processes}
   Probability distributions of L\'evy processes are infinitely divisible, and any infinitely divisible distribution gives rise to a L\'evy processes \cite{sato}. 
   \begin{defin}\label{def:inf_div_distr}
   A distribution $\rho$ on $\bR$ is called SINH-regular of order $\nu$ (resp., $(\nu',\nu)$) and type $((\mum,\mup), \cC, \cC_+)$ if the
    characteristic function of $\rho$ is of the form $e^{i\mu\xi-\psi^0(\xi)}$, where $\psi^0$ satisfies the conditions of Definition
    \ref{def:SINH_reg_proc_1D0} (resp.,  \ref{def:SINH_reg_proc_1D}).
    \end{defin}
    \begin{example}\label{ex:gen_hyper}{\rm
    The characteristic function $F(\xi)$ of a generalized hyperbolic distribution  constructed in \cite{BN1} depends on parameters
    $\al, \be, \de, \la$, where $\al>|\be|$, $\de\ge 0$ and $\la,\mu\in\bR$: 
    \bbe\label{eq:ch_f_GHD}
    F(\xi)= e^{i\mu\xi}\left(\frac{\al^2-\be^2}{\al^2-(\be+i\xi)^2}\right)^{\la/2}\frac{K_\la(\de\sqrt{\al^2-(\be+i\xi)^2})}{K_\la(\de\sqrt{\al^2-\be^2})}.
    \ee
    Here, $K_\la$ is the modified Bessel function of the third kind. From the integral representation
    \bbe\label{eq:modBessel3}
    K_\la(z)=\int_0^\infty e^{-z\cosh t}\cosh(\la t)dt\quad (|\mathrm{arg}\,z|< \pi/2)
    \ee
    and the asymptotic expansion
    \bbe\label{eq:modBessel3as}
    K_\la(z)=\left(\frac{\pi}{2z}\right)^{1/2}e^{-z}\sum_{s=0}^\infty\frac{A_s(\la)}{z^s}\quad 
    (z\to\infty\ \mathrm{in}\ |\mathrm{arg}\,z|<3\pi/2-\eps),
    \ee
    where $\eps>0$ is arbitrarily small (see \cite[Eq. (8.03),(8.04)]{
      Olver97}), it follows that, if $\de>0$,  the distribution is SINH-regular of the same type and order as NIG.

      Note that several classes of infinitely divisible distributions are subclasses
      of the class of generalized hyperbolic distributions. Among subclasses used in Finance, only NIG ($\la=-1/2$, with $\al$ and $\be$ fixed) and VG distributions
     (obtainable in the limit $\de\to 0$, with $\al$ and $\be$ fixed) are closed under convolution; the class of  hyperbolic distributions ($\la=1$)  used in mathematical finance \cite{EK} to construct Hyperbolic Processes is not.       }
    \end{example}
    
    \begin{example}\label{ex:Heston}{\rm The conditional distributions in the Heston model are of order 1 and type
    $((\mum,\mup), \bC\setminus i\bR,\bC\setminus i\bR))$, where $\mum<0<\mup$ can be calculated. See \cite{paraHeston}.
    }
    \end{example}
    
     \begin{example}\label{ex:Aff}{\rm The conditional distributions in a wide class of stochastic volatility models with stochastic interest rates  are of order $\nu\in (0,1]$ and type
    $((\mum,\mup), \cC_{-\pi/4,\pi/4},\cC_{-\pi/4,\pi/4})$, where $\mum<0<\mup$. See \cite{pitfalls}.
    }
    \end{example}

\section{Stieltjes-L\'evy and signed Stieltjes-L\'evy processes}
\label{s:ST_and_sSSL}

In this Section, we do not need the lower bound on $\Re\psi^0$. 
\subsection{Representations of
the L\'evy measure via the characteristic exponent}\label{ss:Levymeasure_via_characteristic exponent}\label{ss:reprfpm1}
 Using the Cauchy integral theorem (for details, see   Section \ref{ss:proof_lem:der_SINH}), we derive 
\begin{lem}\label{lem:der_SINH} Let $\psi$ be analytic in $i(\mum,\mup)+(\cC\cup\{0\})$, where $\mum\le 0\le \mup, \mum<\mup$,  $\cC\cup\{0\}\supset \bR$, and let \eq{upperbound_gen} hold for some $\nu\in (0,2)$. 
Then, 
for any closed cone $\cC'\subset \cC\cup\{0\}$ and interval $[\mumpr,\muppr]
 \subset (\mum\mup)$, there exist $C,c>0$, such that, for $j=0,1,\ldots$, and $\xi\in[\mumpr,\muppr]+\cC'$,
 \bbe\label{der_psi_nu02}
 |\psi^{(j)}(\xi)|\le C j! c^{-j}\left(\prod_{k=0}^j(\nu-k)\right) (1+|\xi|)^{\nu-j}.
 \ee\end{lem}
Using Lemma \ref{lem:der_SINH}, we derive representations for the L\'evy measure $F(dx)=f(x)dx$ in the form of integrals
over regular contours $\cL^\pm\subset i(\mum,\mup)+(\cC\cup\{0\})$ with the wings of $\cL^+$ pointing upwards, and the wings of
$\cL^-$ pointing downwards, for instance, $\cL^\pm=i\theta+(e^{i(\pi-\om)}\bR_+\cup e^{i\om}\bR_+)$,  or $\cL^-=i\om_1+b\sinh(i\om+\bR)$, where $b>0$,
$\theta\in (\mum,\mup), \om_1+b\sinh \om\in (\mum,\mup)$ and $\om>0$ for the sign $-$ (resp., $\om<0$ for the sign ``+") is sufficiently small in absolute value.

\begin{lem}\label{lem:repr_psi} Let the characteristic exponent $\psi^0$ satisfy conditions of Lemma \ref{lem:der_SINH}.  
Then the L\'evy measure is absolutely continuous: $F(dx)=f(x)dx$,  $f\in C^\infty(\bR\setminus\{0\})$, and
\beqa\label{reprf(x)p}
f(x)&=&-\frac{1}{2\pi}\int_{\cL^-} e^{-ix\xi}\psi(\xi)d\xi,\ x> 0,
\\\label{reprf(x)m}
f(x)&=&-\frac{1}{2\pi}\int_{\cL^+} e^{-ix\xi}\psi(\xi)d\xi,\ x< 0.
\eqa
\end{lem}
\begin{proof}
Let $\mum<0<\mup$. Then the L\'evy density exponentially decays at infinity, 
and the differentiation under the integral sign in the L\'evy-Khintchine formula is justified. We obtain
\bbe\label{der_J}
\psi^{\prime\prime\prime}_J(\xi)=\int_\bR e^{ix\xi} ix^3 F(dx),
\ee
where $\psi_J$ is the characteristic function of the pure jump component.
It follows from \eq{der_J} and \eq{der_psi_nu02} that $F(dx)=f(x)dx$ is absolutely continuous,
 $f\in C^\infty(\bR\setminus\{0\})$, and $f$ can be recovered
using the inverse Fourier transform
\bbe\label{reprxf(x)2}
ix^3f(x)=\frac{1}{2\pi}\int_\bR e^{-ix\xi}\psi_J^{\prime\prime\prime}(\xi)d\xi,\ x\neq 0.
\ee
If $x>0$, 
we deform the line of integration downward
\bbe\label{reprx2f(x)p}
ix^3f(x)=\frac{1}{2\pi}\int_{\cL^-} e^{-ix\xi}\psi^{\prime\prime\prime}_J(\xi)d\xi,\ x> 0.
\ee
For any polynomial $P(\xi)$ and $x>0$,
$
\int_{\cL^-}e^{-ix\xi}P(\xi)d\xi=0$, therefore, integrating in \eq{reprx2f(x)p} by parts, we obtain \eq{reprf(x)p}.
The proof of \eq{reprf(x)m} is by symmetry.

If either $\mum=0<\mup$ or $\mum<0=\mup$, we take $\al\in (\mum,\mup)$, and consider 
the Esscher transform of $X$, with the characteristic exponent $\psi(\xi-i\al)-\psi(-i\al)$ and the L\'evy measure
$e^{\al x}F(dx)$. We have 
\bbe\label{reprf(x)pal}
e^{\al x}f(x)=-\frac{1}{2\pi}\int_{\cL^-_\al} e^{-ix\xi}\psi(\xi-i\al)d\xi,\ x> 0,
\ee
where $\cL^-_\al\subset i(\mum-\al,\mup-\al)+(\cC\cup\{0\})$ is a sufficiently regular contour with the wings pointing down.
Shifting the contour and changing the variable $\xi=\xi'+i\al$, we obtain \eq{reprf(x)p}. Eq. \eq{reprf(x)m} is proved similarly.

\end{proof}

\subsection{Representations of the L\'evy measure in the case $\cC=\bC\setminus i\bR$}\label{ss:sSL1}
If  $\cC=\bC\setminus i\bR$, then, under additional  conditions, we can express $f_+=\bfo_{(0,+\infty)}f$ and
$f_-=\bfo_{(-\infty,0)}f$ in terms of integrals over the cuts $i[(-\infty,\mum]$ and $i[\mup,+\infty)$, respectively, w.r.t. to certain  measures $\cG_\pm(dt)=\cG_\pm(\psi; dt)$ (possibly, {\em signed}):
\beqa\label{reprfp1}
f_+(x)&=&\int_{(0,+\infty)}e^{-tx}\cG_+(dt),\ x>0,\\\label{reprfm1}
f_-(x)&=&\int_{(0,+\infty)}e^{t x}\cG_-(dt),\ x<0,\eqa
where $\supp\, \cG_+\subset [-\mum,+\infty)$ (if $\mum=0$, $\cG_+$ has no atom at 0), and
$\supp\, \cG_-\subset [\mup,+\infty)$ (if $\mup=0$, $\cG_-$ has no atom at 0).
In terms of the Laplace transforms $\widetilde{\cG_\pm(dt)}$
 of measures $\cG_\pm(dt)$ 
\bbe\label{reprLTpm}
f_+(x)=\widetilde{\cG_+(dt)}(x),\quad f_-(x)=\widetilde{\cG_-(dt)}(-x). 
\ee
\begin{example}\label{ex:mero}{\rm In the case of HEJD, the $\be$-model and meromorphic processes in general,
we can derive the representations \eq{reprfp1}-\eq{reprfm1} moving the contour of integration in \eq{reprf(x)m} up and in \eq{reprf(x)p} down, and, on crossing each simple pole on the corresponding imaginary half-axis, apply the residue theorem. For calculation of the residues in the $\be$-model,
see Section \ref{ss:calc_res_beta}. The measures $\cG_\pm(dt)=\cG_\pm(\psi; dt)$ are discrete:
\bbe\label{eq:disc_ex1}
\cG_\pm (dt)=\sum_{\al\in \cA_\pm}g^\pm_\al \de_{\al},
\ee
where $\cA_\pm$ are finite or discrete sets with the only accumulation point at $+\infty$. The set $-\cA_+$ (resp., $\cA_-$) is the set of (simple) poles of $\psi$ on $(-\infty,0)$ (resp., $(0,+\infty)$), and
$g^+_\al=\mathrm{Res} (i\psi, -i\al)$, $g^-_\al=-\mathrm{Res} (i\psi, i\al)$ are positive.

If we relax the conditions on the parameters of HEJD and meromorphic model so that some of $g^\pm_\al$ are negative
(but the process is a L\'evy process), then we obtain signed SL-processes. 
MEJD processes are obtained in this way, and one can generalize the class of meromorphic processes in a similar way.
For further generalizations,  see Example \ref{ex:one-sided-mero}.

}
\end{example}
\begin{example}\label{ex:St_repr}{\rm Let $X$ be the one-sided stable L\'evy process of index $\al\in (0,2), \al\neq 1$, with the L\'evy density $f_+(x)=x^{-\al-1}\bfo_{x>0}$ and 
the characteristic exponent  
$\psi_+(\xi)=-\Ga(-\al)(0-i\xi)^\al$. We reduce the integral \eq{reprfp1} to the cut $i(-\infty,0]$:
\beqast
f_+(x)&=&\frac{\Ga(-\al)}{2\pi}\left(\int_{-i\infty-0}^{-0}+\int_{+0}^{-i\infty+0}\right)(0-i\xi)^\al e^{-ix\xi}d\xi\\
&=&-\frac{\Ga(-\al)}{\pi}\int_0^{+\infty}\frac{e^{i\pi\al}-e^{-i\pi\al}}{2i} t^{\al}e^{-tx} dt.\\
\eqast
Thus, \eq{reprfp1} holds with $\cG_+(dt)=\Ga(-\al)\sin(-\pi\al)\pi^{-1} t^\al dt$.
}
\end{example}
\begin{defin}\label{def:Xpm}
$X^\pm=X^\pm_{\cG_\pm}$ denotes the one-sided L\'evy processes given by the generating triplets $(0,0,f_\pm(x)dx)$, where 
 $f_\pm(x):=f_\pm(\cG_\pm;x)$ are defined by \eq{reprfp1}-\eq{reprfm1}. The  characteristic exponents of $X^\pm_{\cG_\pm}$ are denoted $\psi_\pm(\xi):=\psi_\pm(\cG_\pm,\xi)$.
 \end{defin}
 
 Evidently, $\psi_-(\cG;\xi)=\psi_+(\cG,-\xi)$, therefore, it suffices to derive the condition for $\cG$ to define
 the L\'evy densities for $X^+_\cG$ and formula for $\psi_+(\cG;\xi)$; the condition for $X^-_\cG$ is the same, and
 the formula for $\psi_-(\cG;\xi)$ obtains by symmetry.

\begin{lem}\label{lem:condcG1} 
\begin{enumerate}[(a)]

\item
Let $\cG_+\ge 0$. Then
\eq{reprfp1} defines a L\'evy density if and only if
\bbe\label{intcondpp}
\int_{(0,+\infty)} \frac{\cG_+(dt)}{t+t^m}<\infty,
\ee
where $m=3$; the density is completely monotone.
\item
If 
\bbe\label{intcondpm}
\int_{(0,+\infty)} \frac{|\cG(dt)|}{t+t^m}<\infty,
\ee
where $m=3$, and $f_+$ given by  \eq{reprfp1} is non-negative, then $f_+$ is a L\'evy density.
\item
The pure jump process $X^+$ with the L\'evy density $f_+$ is of finite variation iff \eq{intcondpm} holds with $m=2$.
\end{enumerate}
\end{lem}
\begin{proof} 
Substitute \eq{reprfp1} into integrals
$
\int_{(0,1]}x^2 f_+(x)dx,\ \int_{(0,1]}x f_+(x)dx,\ \int_{[1,+\infty)}x^2 f_+(x)dx,
$
and use Fubini's theorem. 
\end{proof}
Measures $\cG(dt)$ satisfying conditions of Lemma \ref{lem:condcG1} admit natural representations in terms of Stieltjes measures, and inherit several important properties  of the latter.

\subsection{Stieltjes measures and functions. Stieltjes transform}\label{ss:SMandST}
\begin{defin}\label{def:St}
{\rm A non-negative measure  $\cG$  on $(0,+\infty)$ is a Stieltjes measure iff 
\bbe\label{StCond}
\int_{(0,+\infty)} (1+t)^{-1}\cG(dt)<\infty.
\ee
We write $\cG\in \SM_0$. The Stieltjes transform $ST(\cG)$ of $\cG$ is given by
\bbe\label{eq:Stf}
ST(\cG)(z)=\int_{(0,+\infty)} (z+t)^{-1}\cG(dt).
\ee
}
\end{defin}
The definition of the Stieltjes transform is standard - see, e.g., (\cite[Defin. 2.1]{schilling_book_Bernstein2012}). 
We introduce the notation $ST(\cG)$ to shorten the formulation of the definitions, statements and proofs below. 

\begin{defin}\label{def:St2} {\rm
If $\cG\in \SM_0$ and $\supp\, \cG\subset [\mu,+\infty), $ where $\mu>0$, we write $\cG\in \SM_\mu$. 

If  $\cG\in \SM_\mu$, $\mu\ge 0$, and $f=ST(\cG_\mu)$, we write 
$f\in \cS_\mu$.
}
\end{defin}
Evidently, for any $\mu'\in [0,\mu]$,
$\SM_\mu\subset \SM_{\mu'}$, $\cS_{\mu'}\subset \cS_\mu$.

The following proposition is immediate from the definition.
\begin{prop}\label{prop:analSM}
Let $\cG\in \SM_\mu$. Then $ST(\cG)$ 
is  analytic  in $\bC\setminus (-\infty,-\mu]$.
\end{prop}
\begin{lem}\label{STliminf}
   Let $\cG\in \SM_0$. Then, 
   for any $\ga\in (0,\pi)$, $ST(\cG)(\rho e^{i\varphi})\to 0$ as $\rho\to+\infty$, uniformly in $\varphi\in [-\ga,\ga]$.
    \end{lem}
   \begin{proof}  There exists $C=C(\ga)>0$ such that, uniformly in $\rho>1$ and $\varphi\in [-\ga,\ga]$,
   \[
   \left|\int_{(0,+\infty)}(t+\rho e^{i\varphi})^{-1}\cG(dt)\right|\le C\int_{(0,+\infty)}(t+\rho)^{-1}\cG(dt)
  \le C\int_{(0,+\infty)}(t+1)^{-1}\cG(dt)<\infty,
  \]
and it remains to apply   the dominated convergence theorem.
\end{proof}
The following lemma is a straightforward reformulation of \cite[Cor. 6.3]{schilling_book_Bernstein2012}
\begin{lem}\label{corr_lim_CBF} Let $\cG\in \SM_0$. 
Then for all positive continuity points $u<v$ of $t\mapsto \cG((-\infty,t])$, 
\bbe\label{eq:repr_sg}
\cG((u,v])=\frac{1}{\pi}\lim_{\eps\to 0+}\int_{[-v,-u)}ST(\cG)(s-i\eps)ds.
\ee
\end{lem}
\begin{lem}\label{lim_0}
Let $\cG\in \SM_\mu$ have no atom at $\mu$. Then
\bbe\label{lim_mum}
\lim_{\eps\to 0+}\int_{|z+\mu|=\eps, \Re z>-\mu}ST(\cG)(z)dz=0.
\ee 
\end{lem}
Proof in Section \ref{ss:proof_of_lim_0}.

\subsection{Definition of SL and sSL processes and examples}\label{ss:def_sSL_SL}
The class $\cS$ of Stieltjes functions (see, e.g., \cite[Defin. 2.1]{schilling_book_Bernstein2012}) is wider than $\cS_0$: $f\in \cS$ if there exist $a_0,a_1\ge 0$ and $\cG\in \SM_0$ such that 
$f(z)=a_0/z+a_1+ST(\cG)(z)$.  For construction of spectrally one-sided L\'evy processes,
$\cS$ is appropriate. Indeed,
any $\cG\in \SM_0$ is the Stieltjes measure of a complete Bernstein function $g$, which appears in 
the Stieltjes representation of  $g$:
\bbe\label{CBFrepr}
 g(z)=a_0+a_1z+zST(\cG)(z)
 \ee
 (see, e.g., \cite[Thm 6.2, Corr. 6.3 and Remark 6.4]{schilling_book_Bernstein2012}).
Evidently, if $a_0=0$, $\psi(\xi)=g(-i\xi)$ is the characteristic exponent of a subordinator.

For construction of more general L\'evy  processes,  class $\cS_0$ is more convenient. 
We start with the following evident proposition.
\begin{prop}\label{prop:equiv}

Measure $\cG$ satisfies \eq{intcondpp} with $m=3$ (resp., with $m=2$) iff there exist $a_2, a_1\ge 0$ and $\cG^0\in \SM_0$
such that 
$\cG(dt)=(a_2t^2 +a_1t)\cG^0(dt)$ (resp., $\cG(dt)=t\cG^0(dt))$.

\end{prop}

\begin{defin}\label{def:sSLmeasure} Let $\mu\ge 0$. 
We say that the measure $\cG$ on $[\mu,+\infty)$ is a  Stiletjes-L\'evy measure (SL measure) of class $\SLM_\mu$ if 
there
exists $a_2,a_1\ge 0$, $a_2+a_1>0$, such that 
$(a_2t^2+a_1t)^{-1}\cG(dt)\in \SM_\mu$.

We say that  $\cG$ is a signed Stiletjes-L\'evy measure (sSL measure)  of class $\sSLM_\mu$ if $\cG$ admits the representation 
$\cG=\cG_1-\cG_2$, where $\cG_j\in \SLM_\mu$, and
 the Laplace transform $\tilde\cG$ of $\cG$ is non-negative on $(0,+\infty)$.
 \end{defin}
 The following evident proposition will be used in the study of Meixner processes in Example \ref{ex:Meixnerdirect0}.
 The proposition admits natural generalizations using maps more general than translations. See Example \ref{ex:subMeixner}.
 \begin{prop}\label{prop:suffsSLM}
 Let $\cG_1\in \SLM_0, \cG_2\in \SLM_\mu$, where $\mu>0$, and let there exist $A\in (0,\mu)$ such that, for any $0\le u<v$,
 $\cG_1((u-A,v-A])\ge \cG_2((u,v])$. Then $\cG_1-\cG_2\in \sSLM_0$.
 \end{prop}
 
 \begin{defin}{\rm Let $\cG^0\in \SM_\mu$, $\cG(dt)=(a_2t^2+a_1t)\cG^0(dt)$, and $X^\pm=X^\pm_\cG$. 
 Then we write $X^\pm\in SL^\pm_\mu$. 
 If $a_1=0$ (resp., $a_2=0$), we write $X^\pm\in SL^{2;\pm}_\mu$
 (resp., $X^\pm\in SL^{1;\pm}_\mu$).

}
\end{defin}
Evidently, if $\mu>0$, one can use a simpler definition of  classes $SL^{2;\pm}_\mu$ instead of the general definition of
classes $SL^\pm_{\mu}$; the statement $X\in SL^{1;\pm}_\mu$ is useful if we wish to indicate that the jump component of $X$  is of finite variation. The following proposition demonstrates that 
$SL^\pm_0\neq SL^{1,\pm}_0\cup SL^{2,\pm}_0$.
\begin{prop}\label{prop:stable_SL}
Let $X^\pm $ be the one-sided stable L\'evy process of index $\al\in (0,2)$, with the L\'evy density $f_\pm(x)=\Ga(\al+1)|x|^{-\al-1}, \pm x>0$. 
Then
\begin{enumerate}[(a)]
\item
if $\al\in (0,1)$, then $X^\pm$ is in $SL^{1;\pm}_0$ but not in $SL^{2;\pm}_0$; 
\item
if $\al\in (1,2)$, then $X^\pm$ is in $SL^{2;\pm}_0$ but not in $SL^{1;\pm}_0$; 
\item
if $\al=1$, then $X^\pm$ is in $SL^\pm_{0}$ but not in $SL^{1;\pm}_0\cup SL^{2;\pm}_0$.
\end{enumerate}
\end{prop}
\begin{proof}
Since
$
\Ga(\al+1)x^{-\al-1}=\int_0^{+\infty} e^{-tx} t^\al dt,
$ the Stieltjes-L\'evy measure is $\cG_+(dt)=t^\al dt$.
\\\noindent
If $\al\in(0,1)$,  $t^{\al-1}dt\in \SM_0$; however, 
$t^{\al-2}dt$  is not in $\SM_0$.
\\\noindent
If $\al\in(1,2)$, $t^{\al-2}dt\in \SM_0$; however, 
$t^{\al-1}dt$  is not in $\SM_0$.
\\\noindent
If $\al=1$, $(t+t^2)^{-1}t^\al dt\in \SM_0$; however, neither
$t^{\al-1}dt$  nor $t^{\al-2}dt$ are in $\SM_0$.
\end{proof}

\begin{example}\label{ex:one-sided-mero}{\rm The natural generalization of one-sided meromorphic processes
is the class of $SL^\pm_0$ processes defined by atomic measures of the form 
\bbe\label{eq:gen_mero}
\cG(dt)=\sum_{\al\in\cA} g_\al \de_\al,
\ee
where $\cA\subset (0,+\infty)$ is a countable set and, for some $a_2, a_1\ge 0$,  
\[
\sum_{\al\in\cA} \frac{g_\al}{(a_2\al^2+a_1\al)(1+\al)}<\infty.
\]
If $\cA$ is finite, we obtain one-sided HEJD processes. If the only point of accumulation of $\cA$ is $+\infty$, then
we obtain meromorphic processes, the $\be$-model in particular. In the $\be$-model, the sequence $\cA$ is uniformly spaced.
Meromorphic model can be used to approximate KoBoL (albeit rather inefficiently, especially if the atoms are uniformly spaced - more efficient approximations can be obtained using non-uniform sequences of atoms).
If we allow for $\cA$ to have an accumulation point at 0, then we can approximate stable L\'evy processes (also inefficiently).
}
\end{example}
\begin{example}\label{ex:KoBoL}{\rm In Section \ref{ss:reprfpm_SL}, we prove that
stable L\'evy processes, KoBoL, VGP, NIG, and NTS  are SL-processes, whereas 
Meixner processes are sSL processes, and derive explicit formulas for the corresponding SL-
and sSL-measures.
}
\end{example}

\begin{thm}\label{thm1:SL}
Let $\cG(dt)=(a_2t^2+a_1t)\cG^0(dt)\in SLM_\mu$ and $X^\pm=X^\pm_\cG$. Then
\begin{enumerate}[(a)]
\item
the characteristic exponent of $X^\pm$ is of the form
\bbe\label{eq:psip2}
\psi_\pm(\xi)=(a_2\xi^2\mp ia_1\xi) ST(\cG^0)(\mp i\xi)\pm i\mu\xi,
\ee
where $\mu=\mu(a_2,a_1,\cG^0)\in \bR$.
If $a_2=0$,  and the L\'evy-Khintchine formula for processes of finite variation is used,
then $\mu= 0$;
\item
$X^\pm\in SL^{1;\pm}_\mu$ are finite variation processes;
\item
if $a_2=0$ and $t\cG^0(dt)\in L_1$, then $X^\pm \in SL^{1;\pm}_1$ are of finite activity;\item
if $\mu>0$, then $SL^\pm_{\mu}=SL^{2;\pm}_\mu$;
\item
if $X^\pm\in SL^\pm_{\mu}$ and $t\cG^0(dt)\in \SM_\mu$, then $X^\pm\in SL^{1;\pm}_\mu$;
\end{enumerate}
\end{thm}
\begin{proof}
To prove (a), we use the L\'evy-Khintchine formula and Fubini's theorem. For details, see Section \ref{ss:proof_thm1:SL}. 
(b) follows from Lemma  \ref{lem:condcG1}, and (c) can be proved in the same fashion. (d) $X^\pm$ can be defined as an element of $SL^{2;\pm}_\mu$ by the measure $(a_2+a_1 t^{-1})\cG^0(dt)\in SM_\mu$. (e) $X^\pm$ can be defined as an element of $SL^{1;\pm}_\mu$ by the measure $(a_2t+a_1)\cG^0(dt)\in SM_\mu$.

\end{proof}
In the following Lemma, by a slight abuse of notation, we use the same label $\psi_\pm$ for the characteristic exponent 
defined by the generating triplet $(\sg^2,b,f_+(x)dx)$, where $f_\pm(x)=f_\pm(\cG;x)$ are defined by 
\eq{reprfp1}-\eq{reprfm1} with $\cG_\pm=\cG(a_2,a_1,\cG^0)$. In the case $a_2=0$, we use the L\'evy-Khintchine formula for jump component of finite variation. The proof of the lemma is immediate from Theorem \ref{thm1:SL}, Lemma  \ref{STliminf} and Proposition \ref{prop:analSM}.

\begin{lem}\label{lem:analinf}
\begin{enumerate}[(a)]
\item
 $\psi_+$ is analytic in $\bC\setminus i(-\infty,-\mu]$, $\psi_-$ is analytic in $\bC\setminus i[\mu,+\infty)$;
 \item
$\forall$\ $\ga\in (0,\pi)$, 
$\psi_\pm(\xi)\sim \frac{\sg^2}{2}\xi^2$ as $(\pm i\cC_\ga\ni)\xi\to \infty$, uniformly in $\mathrm{arg}\,\xi\in [-\ga,\ga]$. 
\item
If  $\sg^2=a_2=0$, then, $\forall$\ $\ga\in (0,\pi)$, 
$\psi_\pm (\xi)\sim -ib\xi$ as $(\pm i\cC_\ga\ni)\xi\to\infty$, uniformly in $\mathrm{arg}\,\xi\in [-\ga,\ga]$ \end{enumerate}

\end{lem}

 \begin{defin}\label{def:sSLprocess} Let $\mum\le0\le\mup$.
 A L\'evy process $X$ is called a signed Stieltjes-L\'evy process (sSL-process) of class $sSL_{\mum,\mup}$ if 
 the L\'evy density of $X$ is of the form 
 \bbe\label{LdenssSL}
 f(x)=\bfo_{(-\infty,0)}(x) \int_{(0,+\infty)}e^{tx}\cG_-(dt)+\bfo_{(0,+\infty)}(x) \int_{(0,+\infty)}e^{-tx}\cG_+(dt),
\ee
where $\cG_-\in \sSLM_{\mup}$ and $\cG_+\in \sSLM_{-\mum}$. If $\cG_-\in \SLM_{\mup}$ and $\cG_+\in \SLM_{-\mum}$, $X$ is called a  Stieltjes-L\'evy process (SL-process) 
of class $SL_{\mum,\mup}$.
\end{defin}
The following theorem is immediate from Definition \ref{def:sSLprocess} and Lemma \ref{lem:analinf}.
\begin{thm}\label{thm:SL-sSLchexp} Let $X\in sSL_{\mum,\mup}$. 
Then 
\begin{enumerate}[(i)]
\item
the characteristic exponent $\psi$ of $X$ is of the form \eq{eq:sSLrepr}, where $\cG^0_+\in SM_{-\mum}, \cG^0_+\in SM_{\mup}$;
\item
$\psi$ admits analytic continuation to $\bC\setminus i((-\infty, \mum]\cup [\mup,+\infty))$;
\item
$\forall$\ $\ga\in (0,\pi/2)$, $\psi(\xi)\sim \frac{\sg^2}{2}\xi^2$ as $\xi\to\infty$, uniformly in
$\mathrm{arg}\,\xi\in[-\ga,\ga]\cup [\pi-\ga,\pi+\ga]$;
\item
if $\sg^2=a^+_2=a^-_2=0$, then, for any $\ga\in (0,\pi/2)$, $\psi(\xi)\sim -ib\xi$ as $\xi\to\infty$, uniformly in
$\mathrm{arg}\,\xi\in[-\ga,\ga]\cup [\pi-\ga,\pi+\ga]$.\end{enumerate}
\end{thm}

\subsection{Representations of  sSL-measures in terms of the characteristic exponent}\label{ss:reprfpm_SL}

 Let 
  $\cG_+=\cG_{+;+}-\cG_{+;-}$, $\cG_{+;\pm}\in \SLM_{-\mum}$ 
 and  $\cG_-=\cG_{-;+}-\cG_{-;-}$, $\cG_{-;\pm}\in \SLM_{\mup}$
 be the Jordan decompositions of measures $\cG_\pm$ in \eq{LdenssSL} Denote by $U_{+;\pm}$ the set of points of continuity of 
 $t\mapsto \cG_{+;\pm}(-\infty,t)$ and set $U_+=U_{+;+}\cap U_{+,-}$. Similarly, define $U_-$.
 \begin{thm}\label{thm:repr_sSL} Let $X$ be of class $sSL_{\mum,\mup}$, with
 the L\'evy density \eq{LdenssSL} and characteristic exponent $\psi$. Then
 \begin{enumerate}[(a)]

\item
for any $\theta\in (\mum,\mup)$ and $x>0$,
 \beqa\label{reprfpsSL}
  f_\pm(\pm x)&=&\frac{1}{\pi}\lim_{\eps\to 0+}\int_{-\theta}^{+\infty} e^{-t x}\Im\psi(\mp(it+\eps))dt;
 \eqa
  \item
for any  $u,v\in U_\pm$,
\bbe\label{reprcGp}
\cG_\pm((u,v])=\lim_{\eps\to 0+}\frac{1}{\pi}\int_{u}^{v} \Im \psi(\mp(it+\eps))dt;
\ee
 \item
 if $\cG_\pm(\{\mp\mu_\mp\})=0$, 
  then \eq{reprfpsSL} holds with $\mp\mu_\mp$ in place of $-\theta$.

 \end{enumerate}
 \end{thm}
\begin{proof} We assume that $\cG_\pm \in \SLM_{\mp\mu_\mp}$; the statements for $\cG_\pm \in \sSLM_{\mp\mu_\mp}$ follow by linearity. By symmetry, it suffices to consider $f_+$ and $\cG_+$.
To prove
 \eq{reprfpsSL}, we use the representation
$\cG_+(dt)=(a_2 t^2+a_1 t)\cG^0_+(dt)$, where $\cG^0_+ \in \SM_{-\mum}$.
Since the characteristic exponent $\psi_-$ of negative jumps and the one of the Brownian motion component are analytic in
$\{\Im\xi<\theta\}$ and polynomially bounded,  we can use \eq{reprf(x)p} with $\psi_+$ given by \eq{eq:psip2}
in place of $\psi$. We fix $\ga\in (0,\pi/2)$, $A>-\mum$ and $\eps>0$,
and deform $\cL^-$ into the union of the following contours: $\cL^-_1=\{e^{i(\pi+\ga)}\rho\ |\ \rho \ge A/\sin \ga\}$; 
 $\cL^-_2=-iA+(-A\cot\ga,-\eps)$; $\cL^-_3=\{\xi\ |\ \mathrm{dist}\,(\xi, i[-A,\theta])=\eps, \Im\xi\ge -A\}$;
 $\cL^-_4=-iA+(\eps,A\cot\ga)$; $\cL^-_5=\{\xi=e^{-i\ga}\rho\ |\ \rho \ge A/\sin \ga\}$. 
On the strength of Theorem \ref{thm:SL-sSLchexp},
 the integrals over $\cL^-_1$ and $\cL^-_5$ tend to 0 as $A\to  +\infty$ (recall that $x>0$). 
 
 Next,  we note that  there exists $C>0$ independent of $R$ and $\eps$ such that 
  \beqast
  \left| \left(\int_{\cL^+_2}+\int_{\cL^+_4}\right)e^{-ix\xi}\psi(\xi)d\xi\right|
&\le& C e^{-xA}\int_{(-\mum,+\infty)}\cG^0_+(dt)(a_2t^2+a_1t)\\
&&\times \left|\left( \int_{-A\cot\ga}^{-\eps}+\int_{\eps}^{A\cot\ga} \right)\frac{dy}{t-A+iy}\right|
\eqast
Straightforward  calculations prove that $\left( \int_{-A}^{-\eps}+\int_{\eps}^A \right)\frac{dy}{t-A+iy}$ is uniformly bounded w.r.t.
$t,A,\eps$, and bounded by $C/(A+t)$ on $(-\mum+2A,+\infty)$, uniformly in $\eps>0$. Taking into account that $\cG^0_+\in \SM_0$ and $x>0$,  we conclude that
the sum of integrals over $\cL^-_2$ and $\cL^-_4$ tends to 0 as $A\to  +\infty$, uniformly in $\eps>0$.

It remains to consider the limit of the integral over $\cL^+_3=\cL^+_3(A,\eps)$ as $\eps\to 0$ and $A$ fixed. Since $\overline{\psi(\xi)}=\psi(-\bar\xi)$ and $\psi$ is real and continuous at any point of $i(\mum,\theta)$, we have
\beqast
&&-\frac{1}{2\pi}\lim_{\eps\to 0+}\int_{\cL^+_3(A,\eps)}e^{-ix\xi}\psi_+(\xi)d\xi\\
&=&\frac{1}{2\pi}\lim_{\eps\to 0+}\left(\int_{-A}^{\theta} e^{(t+i\eps) x} (-\psi_+(it-\eps)i dt
+\int_{-A}^{\theta} e^{(t-i\eps) x}\psi_+(it+\eps))i dt
\right)\\
&=&\frac{1}{\pi}\lim_{\eps\to 0+}\int_{-\theta}^A e^{-t x}(1+g_1(t,\eps)) \Im\psi_+(-it-\eps)dt\\
&=&\frac{1}{\pi}\lim_{\eps\to 0+}\int_{-\theta}^A e^{-t x} e^{-tx}(1+g(t_2,\eps))(a_2t^2+a_1t)\Im ST(\cG^0_+)(-it-\eps)dt,
\eqast
where $g_j(t,\eps), j=1,2,$ are continuous, and $g_j(t,x)=o(\eps), j=1,2,$ as $\eps\to 0$, uniformly in $t\ge 0$.
We can choose $A$ to be a point of continuity of $t\mapsto \cG^0_+(\theta, t)$. Next, if $\mum<0$, we can choose
$\theta\in (\mum, 0)$. Then, by Lemma \ref{corr_lim_CBF},
the limit above is 
\[
\int_{-\theta}^{A} e^{- tx}(a_2t^2+a_1t)\cG^0_+(dt)=\int_{-\theta}^{A} e^{- tx}\cG_+(dt).
\]
We have proved that $f_+(x)-\int_{-\theta}^{A} e^{- tx}\cG_+(dt)\to 0$ as $(U_+\ni) A\to +\infty$, hence, \eq{reprfpsSL} holds.
Using the Esscher transform, we can reduce the case $\mum=0$ to the case $\mum<0$.
The argument above demonstrates that  the Laplace transform of $\cG_+$ and the Laplace transform of the measure
defined by the RHS of \eq{reprcGp} coincide, which proves \eq{reprcGp}.
If $\cG_+(\{-\mum\})=0$, equivalently, $\psi$ does not have a pole at $i\mum$, then we use \eq{lim_mum} to conclude 
 that we can pass to the limit $\theta\downarrow \mum$ in the proof of \eq{reprfpsSL} and \eq{reprcGp}.

\end{proof}
The  following verification theorem allows us to prove that stable L\'evy processes, KoBoL, VGP, NIG, NTS 
and Hyperbolic processes
are SL-processes whereas Meixner processes are sSL-processes but not SL-processes; the conditions of the theorem can be relaxed.

\begin{thm}\label{thm:SLverification}
Let the characteristic exponent $\psi$ of a L\'evy process satisfy the following conditions:
\begin{enumerate}[(i)]
\item
$\psi$ is analytic in $\bC\setminus i((-\infty,\mum]\cup [\mup,+\infty))$, where $\mum\le0\le \mup$, and $\mum<\mup$;
\item
$\exists$ $\nu<2$, $\de>0$ and $C>0$ s.t. $\forall$\ $\xi\in \bC\setminus i((-\infty,\mum]\cup [\mup,+\infty))$, 
the characteristic exponent of the jump component admits the bound
\bbe\label{upbound}
|\psi_J(\xi)|\le C(|\xi|^\nu+|\xi-i\mum|^{-1+\de}+|\xi-i\mup|^{-1+\de});
\ee
\item
$\forall$\ $\be\in (-\infty,\mum)\cup (\mup,+\infty)$, the limit $\Im\psi(i\be+ 0)$ exists, a.e.

\end{enumerate}
Then 
\begin{enumerate}[(a)]
\item
$\cG_\pm (dt)=\frac{1}{\pi}\Im\psi(\mp(it+ 0))dt$  are sSL-measures;
\item
$X$ is a sSL-process, with the L\'evy density given by \eq{LdenssSL};

\item if $\Im\psi(it+ 0)\ge 0$ for $t>\mup$, and $\Im\psi(-it- 0)\ge 0$ for $t>-\mum$,
then $\cG_\pm$ are SL-measures, and $X$ is an SL-process.
\end{enumerate}
\end{thm}
\begin{proof}
(a) It suffices to notice that $\overline{\psi(\xi)}=\psi(-\bar\xi)$, and  functions $t\mapsto (1/\pi)\Im\psi(\pm(it+0))$ are measurable and satisfy the bound \eq{upbound}.
(b) The proof is a straightforward simplification of the proof of Theorem \ref{thm:repr_sSL}, the bound \eq{upbound}
and the dominated convergence theorem being used.
(c) is immediate from (a) and the definition of sSL- and SL-measures.
\end{proof}

\begin{example}\label{ex:NTSdirect0}{\rm 
Let $\psi=\psi^0$ be given by \eq{NTS2} (the case of NIG and  NTS processes).   Clearly, $\psi(\xi)$
admits the bound \eq{upbound} with $\de=1$. Next, we represent $\psi$ 
in the form
 \begin{equation}\label{NTS30}
\psi(\xi)=\de[(\al-\be-i\xi)^{\nu/2}(\al+\be+i\xi)^{\nu/2}-(\al^2-\be^2)^{\nu/2}],
\end{equation}
and derive
\beqa\label{NTSpsiJp}
 \Im\psi(it+0)&=&\de\sin(\pi\nu/2) ((t-\be)^2-\al^2)^{\nu/2}, \ t>\mup:=\al+\be,\\\label{NTSpsiJm}
\Im\psi(it-0)&=&\de\sin(\pi\nu/2) ((t+\be)^2-\al^2)^{\nu/2}, \ t<\mum:=-\al+\be.
 \eqa Hence, $X$ is an SL process, which is a mixture of two independent one-sided SL-processes.  

}
\end{example}

\begin{example}\label{ex:KoBoLdirect0}{\rm 
Let $\psi=\psi^0$ be given by \eq{KBLnupnumneq01} (the case of KoBoL with the positive and negative
densities of order $\nu_\pm\in(0,2)\setminus\{1\}$).  Essentially the same argument as in Example \ref{ex:NTSdirect0} proves that  $\psi$ is the characteristic exponent of an SL process, and
\beqa\label{KBLpsiJ02p}
 \Im\psi(it+0)&=& -c_-\Ga(-\num)\sin(\pi\num)(t-\lp)^{\num}, \ t>\lp,\\\label{KBLpsiJ02m}
\Im\psi(-it-0)&=& -c_+\Ga(-\nup)\sin(\pi\nup)(t+\lm)^{\nup}, \ t>-\lm.
 \eqa 
 If $\nu_\pm=1$, then, using \eq{KBL1p}-\eq{KBL1m}, we derive 
 \beqa\label{KBLpsiJ1p}
 \Im\psi(it+0)&=& \pi c_-(t-\lp), \ t>\lp,\\\label{KBLpsiJ1m}
\Im\psi(-it-0)&=& \pi c_+(t+\lm), \ t>-\lm.
\eqa
}
\end{example}

\begin{example}\label{ex:VGPdirect0}{\rm 
Let $\psi=\psi^0$ be given by \eq{KBL_0+}.  We have 
\bbe\label{VGpsi}
\Im\psi(it\pm 0)=\pi c_\mp, \  \pm t>\pm \la_\pm,\ee
and  essentially the same argument as in Example \ref{ex:NTSdirect0} proves that  $\psi$ is the characteristic exponent of an SL-regular process. The only difference is that $\psi$ has the logarithmic singularities at $\la_\pm$, hence, the bound \eq{upbound} holds
with $\de\in (0,1)$.
}\end{example}

\begin{example}\label{ex:Meixnerdirect0}{\rm 
Let $\psi^0$ be given by \eq{eq:Meixner_ch_exp}.  For $\eps\neq 0$, and $t\in \bR$, we have 
\beqast
\Im\psi^0(it+\eps)&=&2\de\Im\ln[\cosh(a\eps/2+i(at-b)/2))]\\
&=&2\de\Im\ln\left[\frac{e^{a\eps/2}+e^{-a\eps/2}}{2}\cos\frac{at-b}{2}+i\frac{e^{a\eps/2}-e^{-a\eps/2}}{2}\sin\frac{at-b}{2}\right].
\eqast
If $\cos\frac{at-b}{2}>0$,  $\Im\psi(it\pm 0) =0$. If $\cos\frac{at-b}{2}<0$ and $\pm \sin\frac{at-b}{2}>0$,
$\Im\psi(it+ 0) =\pm 2\de \pi$ and $\Im\psi(it-0) =\mp 2\de \pi$. Hence, Meixner processes are sSL-processes but not SL-processes.
The sSL-measures are $\cG_\pm(dt)=g_\pm(t)dt$, where, for $t>0$,
\beqa\label{gp_Meixner}
g_+(t)&=&2\de\sum_{k=0}^{+\infty}\left(\bfo_{a^{-1}((4k+1)\pi+b,(4k+2)\pi+b)}(t)
-\bfo_{a^{-1}((4k+2)\pi+b,(4k+3)\pi+b)}(t)\right),\\
\label{gm_Meixner}
g_-(t)&=&2\de\sum_{-\infty}^{-1}\left(-\bfo_{a^{-1}((4k+1)\pi+b,(4k+2)\pi+b)}(-t)
+\bfo_{a^{-1}((4k+2)\pi+b,(4k+3)\pi+b)}(-t)\right).
\eqa
}\end{example}

The proof of the following theorem is essentially the same as the one of Theorem \ref{thm:SLverification}.
The difference is that we drop an assumption that $\psi$ is a characteristic exponent of a L\'evy process.
\begin{thm}\label{thm:SLverification2}
Let function $\psi$ satisfy the following conditions:
\begin{enumerate}[(i)]

\item
$\psi$ is analytic in $\bC\setminus i((-\infty,\mum]\cup [\mup,+\infty))$, where $\mum\le0\le \mup$, and $\mum<\mup$;
\item
$\psi(0)=0$ and $\overline{\psi(\xi)}=\psi(-\bar\xi)$, for all $\xi\in i(\mum,\mup)+(\bC\setminus i\bR)$;
\item
conditions (i)-(iii) of Theorem \ref{thm:SLverification} hold;

\item
the Laplace transforms of the measures 
$\cG_\pm(dt)=\frac{1}{\pi}\Im\psi(\mp(it+ 0))dt$ are non-negative.
\end{enumerate}
Then the conclusions of Theorem \ref{thm:SLverification} are valid.
\end{thm}
\begin{rem}{\rm
Theorem \ref{thm:SLverification2} allows one to prove that the functions $\psi$ in the examples above are
 characteristic exponents without knowing this fact in advance. In  the case of the Meixner processes,
condition (iv) follows from \eq{gp_Meixner}-\eq{gm_Meixner} and Proposition \ref{prop:suffsSLM}; in the other examples, the measures are non-negative, hence, (iv) is satisfied.
}
\end{rem}

 \subsection{Regular SL and sSL processes and distributions}\label{s:regST_and_sSSL}
\begin{defin}{\rm An sSL-process (resp., SL-process) $X$ is called regular sSL-process
(resp., regular SL-process)  if $X$ is SINH-regular.
}
\end{defin}
\begin{prop}\label{prop:SL_SINH} Let
$X\in sSL_{\mum,\mup}$ and \eq{lowerbound_gen} hold. Then $X$
 is SINH-regular of type   $((\mum,\mup);\bC\setminus i\bR; \cC_+)$, and there exists  $\nu\in [\nu',2]$
 such that the order of $X$ is $(\nu',\nu)$.
\end{prop}
\begin{proof} By Lemma 
\ref{lem:analinf}, all conditions of Definition \ref{def:SINH_reg_proc_1D} are satisfied bar \eq{lowerbound_gen}.  

\end{proof}
\begin{example}\label{ex:reg_SL}{\rm  KoBoL, VGP, NIG and NTS 
are regular SL-processes whereas Meixner processes are regular sSL-processes.}
\end{example}
\begin{defin}\label{def:inf_div_distr_SL}
   A distribution $\rho$ on $\bR$ is called an SL (resp., sSL; SL-regular; sSL-regular) distribution if
    characteristic function of $\rho$ is of the form $e^{i\mu\xi-\psi^0(\xi)}$, where $\psi^0$ is the characteristic exponent of
    SL (resp., sSL; SL-regular; sSL-regular) process.
        \end{defin}
\begin{example}\label{ex:GenHyper_SL}{\rm In Section \ref{hyper_SL_calc}, we prove
     \begin{prop}\label{prop:GenHyper_SL} Let  $\sg(dx)$  a Generalized Hyperbolic distribution
 with the characteristic function  given by \eq{eq:ch_f_GHD}. Then
    \begin{enumerate}[(i)]
    \item
    $\sg(dx)$ is a regular sSL distribution;
    \item
    if $\la\in [-2,1]$, $\sg(dx)$ is a regular SL distribution;
    \item
    if $\la>1$, $\sg(dx)$ is not an SL distribution;
    \item
    if $\la<-2$, then, for sufficiently small $\de$ and/or $\al-|\be|$, $\sg(dx)$ is not an SL distribution.
    \end{enumerate}
    \end{prop}
   }\end{example}
     In particular, hyperbolic distributions $(\la=1$) and processes are SL-regular.

In the general case, verification of \eq{lowerbound_gen} for $\psi$ is reducible to verification of \eq{lowerbound_gen} for a spectrally one-sided
pure jump SL process $X$.
We formulate sufficient conditions on the measure $\cG\in \SM_\mu$ in the representation
$(a_2t^2+a_1t)\cG(dt)$ of the SLM measure of $X$.

 If $\cG(dt)$ is absolutely continuous:
$\cG(dt)=g(t)dt$, we assume that there exists $\al<2$ and $c, C>0$ such that
\bbe\label{eq:bound_SL_SINH}
ct^{\al}\le g(t)\le Ct^{\al},  t>0.
\ee
If $\cG(dt)$ is  has atoms and/or the density is unbounded, then an analog of \eq{eq:bound_SL_SINH} is more complicated.
We assume that there exist $C,c>0$ and $\rho_0\ge 1$ such that for $k=0,1$, and any $\rho\ge \rho_0$
\bbe\label{eq:bound_SL_SINH_at}
c\rho^{\al+1}\int_{\mu/\rho}^{+\infty} \frac{t^{k+\al}dt}{t^2+1}\le \int_{\mu/\rho}^{+\infty}\frac{t^k\cG(\rho dt)}{t^2+1}\le C\rho^{\al+1}\int_{\mu/\rho}^{+\infty}\frac{t^{k+\al}dt}{t^2+1},\ k=0,1,\ \rho\ge \rho_0.
\ee
Clearly, \eq{eq:bound_SL_SINH} implies \eq{eq:bound_SL_SINH_at}.

\begin{thm}\label{thm:one-sided_SL_SINH} Let there exist $\al\in [-1,0)$ such that  \eq{eq:bound_SL_SINH_at} holds.
Then
\begin{enumerate}[(a)]
\item
if  $\psi_\pm(\xi)=\xi^2ST(\cG)(\mp i\xi)\mp i\mu\xi$ and  $\al\in (-1,0)$, then $X^\pm$ is SINH-regular of order $(\nu,\nu)$, where $\nu=\al+2$, and the interior of the cone $\cC_+$ contains $\bR\setminus 0$;
\item
if  $\psi_\pm(\xi)=\xi^2ST(\cG)(\mp i\xi)\mp i\mu\xi$ and $\al=-1$, then $X^\pm$ is SINH-regular of order $(1,1+)$. The cone
$\cC_+$ is adjacent to $\bR$. For $\psi_+$,  $\cC_+\subset\{\Im \xi> 0\}$, and for $\psi_-$,  $\cC_+\subset\{\Im \xi< 0\}$;
\item
if $\psi_\pm(\xi)=\mp i \xi ST(\cG)(\mp i\xi)$ and  $\al\in (-1,0)$, then $X^\pm$ is SINH-regular of order $(\nu,\nu)$, where $\nu=\al+1$, and the interior of the cone $\cC_+$ contains $\bR\setminus 0$;
\item
if $\psi_\pm(\xi)=\mp i \xi ST(\cG)(\mp i\xi)$ and  $\al=-1$, then $X^\pm$ is SINH-regular of order $(0+,0+)$, and the interior of the cone $\cC_+$ contains $\bR\setminus 0$. 
\end{enumerate}
\end{thm}
\begin{proof} Proof in Section \ref{ss:proof_thm:one-sided_SL_SINH}. \end{proof}

A natural class of discrete measures $\cG=\sum_j p_j\de_{t_j}$
can be defined by sequences $\{t_j\}$ that accumulate to $+\infty$ (examples are meromorphic processes and the $\be$-model) 
and/or 0. The measures that accumulate both to 0 and $+\infty$ can be used to approximate $\cG(dt)$ of stable L\'evy processes,
hence, approximate the characteristic exponents of stable L\'evy processes. 
 One can formulate explicit conditions on  sequences $p_j, t_j$, $j\in \bZ_{++}$ or $j\in \bZ$ which imply \eq{eq:bound_SL_SINH_at}.
 \begin{example}\label{ex:SL_D1} {\rm Let there exist positive constants $c_1, c_2, C_1, C_2$ such that, for $j\in \bZ_{++}$,
 $c_1  \le t_{j+1}-t_j\le C_1 $ and $c_2 t_j^{\al}\le p_j\le C_2 t_j^{\al}$. Then $\cG=\sum_{j=1}^{+\infty} p_j\de_{t_j}$ satisfies \eq{eq:bound_SL_SINH_at}.
  }
 \end{example}
  \begin{example}\label{ex:SL_D2} {\rm Let there exist positive constants $c_1, c_2, C_1, C_2$ such that, for $j\in \bZ_{++}$,
 $c_1 t_j \le t_{j+1}-t_j\le C_1 t_j$ and $c_2 t_j^{\al-1}\le p_j\le C_2 t_j^{\al-1}$. Then $\cG=\sum_{j=1}^{+\infty} p_j\de_{t_j}$ satisfies \eq{eq:bound_SL_SINH_at}.
  }
 \end{example}
  \begin{example}\label{ex:SL_D3} {\rm Let there exist positive constants $c_1, c_2, C_1, C_2$ such that, for $j\in\bZ$,
 $c_1 t_j \le t_{j+1}-t_j\le C_1 t_j$ and $c_2 t_j^{\al-1}\le p_j\le C_2 t_j^{\al-1}$. Then $\cG=\sum_{j\in\bZ} p_j\de_{t_j}$ satisfies \eq{eq:bound_SL_SINH_at}.
  }
 \end{example}

 
\section{SL-processes: absence of solutions of equation $q+\psi(\xi)=0$ on $\bC\setminus i\bR$}\label{s:sol_SL}
In this section, we prove that SL-processes enjoy the following property:
\bbe\label{eq:no_zeros}
{\rm for\ any}\ q>0,\ {\rm equation}\ q+\psi(\xi)=0\ {\rm has\ no\ solution\ in}\ \bC\setminus i\bR.
\ee
The reader can easily
construct examples  for which this important property does not hold. 
 The simplest example is a spectrally one sided
L\'evy process with $f(x)=xe^{-x}\bfo_{(0,+\infty)}(x)$ and the characteristic exponent $\psi(\xi)=-i\mu\xi+1-(1-i\xi)^{-2}$:
for a given $q>0$, we can choose $\mu$ so that \eq{eq:no_zeros} fails.
However, this process is not an sSL-process.

\subsection{A counter-example: Meixner processes}\label{ss:counterexample}
The following counter-example shows that the property \eq{eq:no_zeros} may fail is $X$ is an sSL-process.

\begin{example}\label{ex:Meixnerdirect_roots}{\rm 
Let $\psi$ be given by \eq{eq:Meixner_ch_exp}, and 
let $q>0$. Set $z=\exp[(a\xi-ib)/2]$, $2Q=\exp(\cos(b/2)-q/(2\de))$, and 
consider the equation 
 $z+1/z-2Q=0$. The solutions are $z_\pm= Q\pm (Q^2-1)^{1/2}$, hence, for any 
 $k\in \bZ$, $\xi_{k;\pm}=(2/a)[\ln z_\pm+i((b/2+2k\pi)]$ are solutions of the equation $\psi(\xi)+q=0$.
Thus, the  equation  $\psi(\xi)+q=0$ has (infinitely many) solutions in $\bC\setminus i\bR$, whereas
in the case of SL-processes, there are none.
}\end{example}
We do not know if \eq{eq:no_zeros} fails for any sSL-process which is not an SL-process.
  \vskip0.1cm \noindent
{\sc Hypothesis SLzeros.} If the jump component $\psi_J(\xi)$ of the characteristic exponent of a L\'evy process $X$ admits analytic continuation to
$(\bC\setminus \bR)\cup\{0\}$ and there exist $q>0$, $\sg^2\ge 0$ and $b\in \bR$ such that the equation 
$q+\frac{\sg^2}{2}\xi^2-ib\xi+\psi_j(\xi)=0$ has a solution in $\bC\setminus i\bR$,
then $X$ is not a SL  process.

\subsection{Main theorem} 
\begin{thm}\label{thm:zeros_SL} Let $\psi$ be the characteristic exponent of a Stieltjes-L\'evy process. 
Then
\begin{enumerate}[(a)]
\item
the property \eq{eq:no_zeros} holds;
\item
if the SL-measures $\cG_\pm$ defining $\psi$ are supported on $[-\mum,+\infty)$ and $[\mup,+\infty)$, respectively, and $\mum<\mup$, then, on each interval $i(\mum,0)$, $i(0,\mup)$, the equation $q+\psi(\xi)=0$ has either 0 or 1 solution;
 \item
 the solution on $i(\mum,0)$ exists iff $\mum<0$ and $\psi(i(\mum+0))+q<0$;
  \item
 the solution on $i(0,\mup)$ exists iff $\mup>0$ and $\psi(i(\mup-0))+q<0$.
 \end{enumerate}
 \end{thm}
 \begin{rem}{\rm To obtain one-sided analogs, set either $\mum=-\infty$ or $\mup=+\infty$. In Section
\ref{ss:proof_of_Lemma_lem:lim0psi}, we give simple sufficient conditions for zeros to exist for any $q$. }
 \end{rem}
 To prove
 (b)-(d), it suffices to notice that  if $\mum<\mup$, the function $(\mum,\mup)\ni t\mapsto \psi(it)\in \bR$ is concave, and $q+\psi(0)=q>0$. To prove (a), we note that 
the natural domain for a Stiltjes function $ST(\cG)\in\cS_\mu$ is 
$\bC\setminus (-\infty,-\mu]$, and rewrite the equation $\psi(\xi)+q=0$, $\xi\not\in i\bR$, in the form $F(z)=
F(\sg^2,\mu, a^+_2,a^+_1,a^-_2,a^-_1, \cG_+,\cG_-;z)=0$, $z\not\in \bR$, where $\sg^2\ge 0, \mu\in \bR, q>0, 
a^\pm_2\ge 0, a^\pm_1\ge0,  \cG_\pm\in SM_0$, and 
    \bbe\label{eq:defF}
 F(z)=- \frac{\sg^2}{2}z^2+ \mu z+(-a^+_2z^2+a^+_1z)ST(\cG_+)(z)+(-a^-_2z^2-a^-_1z)ST(\cG_-)(-z)+q.
  \ee
This, (a) is equivalent to the following theorem.
\begin{thm}\label{thm:zeroF}
Equation $F(z)=0$ has no solution in $\bC\setminus \bR$.
\end{thm}

\subsection{Proof of   Theorem \ref{thm:zeroF}: preliminaries}\label{prelim_proof_zeros}
We use the trivial observation
\bbe\label{apm12}
 -a^\pm_2z^2\pm a^\pm_1z<0,\  \pm z<0,
 \ee
  and auxiliary propositions, which list several properties of the Stieltjes measures.
 The following simple technical lemma is proved in Section \ref{ss:proof of Lemma lem:limcuts}.
  \begin{lem}\label{lem:limcuts} 
   Let 
   $\cG\in \SM_\mu$, where $\mu\ge 0$, and $b>\mu$ satisfy the following condition:
     $\exists$ $c>0$ and $\de\in (0, b)$ such that for any measurable non-negative $g$,
   \bbe\label{lowboundPsi}
   \int_{(b-\de,b+\de)}g(t)\cG(dt)\ge c\int_{(b-\de,b+\de)}g(t)dt.
   \ee
   Then for any $\de'\in (0,\de)$, there exist $\eps_0>0$ and $C>0$ s.t. for all $b'\in (b-\de',b+\de')$ and  $\eps\in (0,\eps_0)$, $\Im ST(\cG)(-b'+i\eps)\le- \pi c+C\eps$ and
   $\Im ST(\cG)(-b'-i\eps)\ge \pi c-C\eps$ .
   \end{lem}

   Let $D=\supp\, \cG$. The function $ST(\cG)$ below is the analytic continuation of 
   $ST(\cG)$ from either $\{\Im\xi\le 0\}\setminus(-D)$ or $\{\Im\xi\ge 0\}\setminus(-D)$ to a sufficiently large simply connected open set.
    The following two lemmas can be formulated and proved in a somewhat more general form, replacing the equality $\cG(dt)=cdt$ with the inequality as in Lemma \ref{lem:limcuts}.
    \begin{lem}\label{lem:limcutsup} 
   Let 
   $\cG\in \SM_\mu$, where $\mu\ge 0$,  $a>\mu$, and $\de\in (0, a)$ satisfy the following conditions:
   \begin{enumerate}[(i)]
   \item
   $\supp\,\cG\cap (a-\de,a)=\emptyset$;
   \item
   the restriction of $\cG(dt)$ on $(a,a+\de)$ is $c dt$, where $c>0$.
   \end{enumerate}
      Then, uniformly in $\varphi\in [-\pi/2,\pi/2]$, as $\eps\downarrow 0$,
     \bbe\label{eq:limcutsup}
     ST(\cG)(-a+\eps e^{i\varphi})= c\ln (1/\eps)  +O(1).
     \ee
      \end{lem}
     \begin{proof} We have
     \beqast
     \int_{(0,+\infty)}(t-a+\eps e^{i\varphi})^{-1}\cG(dt)&=&c\int_a^{a+\de} (t-a+\eps e^{i\varphi})^{-1}dt+O(1)
     = c\ln\frac{\de+\eps e^{i\varphi}}{\eps e^{i\varphi}}+O(1).
     \eqast
   
     \end{proof}
     
     \begin{lem}\label{lem:limcutsdown} 
   Let 
   $\cG\in \SM_\mu$, where $\mu\ge 0$,  $b>\mu$, and $\de\in (0, b)$ satisfy the following conditions:
   \begin{enumerate}[(i)]
   \item
   $\supp\,\cG\cap (b,b+\de)=\emptyset$;
   \item
   the restriction of $\cG(dt)$ on $(b-\de,b)$ is $c dt$, where $c>0$.
   \end{enumerate}
      Then, uniformly in $\varphi\in [-\pi/2,\pi/2]$, as $\eps\downarrow 0$,
     \bbe\label{eq:limcutsdown}
     ST(\cG)(-b-\eps e^{i\varphi})= -c\ln (1/\eps) +O(1).
     \ee
     \end{lem}
     \begin{proof} We have
     \beqast
     \int_{(0,+\infty)}(t-b-\eps e^{i\varphi})^{-1}\cG(dt)&=&c\int_{b-\de}^b (t-b-\eps e^{i\varphi})^{-1}dt+O(1)
     = c\ln\frac{-\eps e^{i\varphi}}{-\de-\eps e^{i\varphi}}+O(1).
      \eqast
     
     \end{proof}

\subsection{Proof of   Theorem \ref{thm:zeroF}}
    The proof is by contradiction. Assume that a solution $z_0\in \bC\setminus \bR$ exists; since $\overline{ST(\cG)(z)}=ST(\cG)(\bar z)$, we may assume that $\Im z_0<0$.
   Since $z_0$ is in the domain of analyticity, the solution does not disappear after a sufficiently small perturbation of $\sg^2$
   and sufficiently small perturbations of $\cG_\pm$  in the norm $\|\cG\|=\int_{(0,+\infty)}(1+t)^{-1}\cG(dt)$. Hence, we may assume
   that the following three conditions hold: 1) $\sg^2>0$; 2) each of  $\cG^\pm$ is 
   an absolutely continuous measure  $\sum_{j=1}^{N_\pm} a^\pm_j\bfo_{(a^\pm_j,b^\pm_j)}(t)dt$, where $[a_j,b_j]\subset (0,+\infty), j=1,2,\ldots, N_\pm,$ are non-intersecting intervals;
   3)
   $\forall k$, $\be^\pm_k\not\in\cup_j [a^\pm_j,b^\pm_j]$. Evidently, $F$ admits analytic continuation from $\{\Im\xi< 0\}$ to a simply connected open set containing
   $\{\Im\xi\le 0\}\setminus (D_+\cup D_-)$, where $D_+=\cup_l[-b^+_l,-a^+_l]$ and $D_-=\cup_m[a^-_m,b^-_m]$.
  
   It follows from Lemmas \ref{lem:limcuts}- \ref{lem:limcutsdown}  and the equality
   $\overline{F(z)}=F(\bar z)$
   that, if $\eps>0$ is sufficiently small,
     $F$ has no zeros in the $\eps$-neighborhood of $D_+\cup D_-$. Since $\sg^2>0$,  $F(z)\to -\infty$ as $(\bR\ni)z\to\pm\infty$,
  therefore,  the number of intervals and zeros of $F$ in   
 $\bR\setminus (D_+\cup D_-)$ is finite. Let $\{-\al^+_j\}$ (resp, $\{\al^-_j\}$) be the set of all zeros on $(-\infty,0)$ (resp., $(0,+\infty)$.)
 For a sufficiently small $\eps>0$, we define the contour $\cL_\eps\subset \{\Im z\le 0\}$ as the union of the following sets:
 \begin{enumerate}[(1)]
 \item
 $\cL_{\eps,\infty}=\{z=-i(1/\eps)e^{i\varphi}\ |\ -\pi/2< \varphi<\pi/2\}$;
 \item
 $\cL_{\eps,0}=\{z\in \bR\ |\ |z|\le 1/\eps, \mathrm{dist}\, (z, D)\ge \eps\}$;
 \item
 $\cL_{\eps,1}=(\cup_j [-b^+_j,-a^+_j]-i\eps)\cup (\cup_\ell [b^-_\ell,b^-_\ell]-i\eps)$;
  \item
 semi-circles in the lower half-plane, of radius $\eps$, around each point $ -\al^+_m, \al^-_n$;
 \item
 quarter-circles in the lower half-plane, of radius $\eps$, with centers at $-a^+_j, -b^+_j, a^-_\ell, b^-_\ell$,  connecting $\cL_{\eps,0}$ 
 and $\cL_{\eps,1}$.
 \end{enumerate}
 We will prove that is $\eps>0$ is sufficiently small, then the winding number  $\mathrm{ind}(\cL, 0,0; F)$ of the curve 
$\{F(z)\ |\ z\in \cL_\eps\}$, where $z$ runs along $\cL_\eps$ in the positive direction, 
is 0. The winding number being equal to the number of zeros in the domain bounded by $\cL_\eps$, and $\eps>0$ being arbitrary small,
the number of zeros of $F$ in the lower half-plane is 0 as well, contradiction.

Recall that for $z_1, z_2$ on $\cL_\eps$, and $\cL_\eps(z_1,z_2)$ the part of $\cL_\eps$ with $z_1, z_2$ being the end points,
\[
\mathrm{ind}(\cL_\eps, z_1,z_1; F)=\frac{1}{2\pi}\int_{\cL_\eps(z_1,z_2)}d\,\mathrm{arg}\, F(z).
\]
As $\rho\to +\infty$, $F(-i\rho e^{i\varphi})\sim \frac{\sg^2}{2}\rho^2e^{2i\varphi}$, hence, if $\eps>0$ is sufficiently small
so that $D_+\cup D_-\subset \bR\setminus \{|z|\ge 1/\eps\}$, $\mathrm{arg} F(z)$ increases from $-\pi$ to $\pi$ as $z$ moves along $\cL_{\eps,\infty}$. Thus,
$\mathrm{ind}(\cL_\eps, -1/\eps, 1/\eps; F)=1$.
Next,  functions $ST(\cG_-)(-z)$, $z$ and $z^2$ are continuous in a neighborhood of $(-\infty,0)$, 
and real-valued on $(-\infty,0)$.
 Hence, 
  it follows from \eq{apm12} and Lemmas \ref{lem:limcuts}-\ref{lem:limcutsdown}, 
   that there exist $\eps_0,c>0$ such that, for each $\eps\in (0,\eps_0]$, 
 \begin{enumerate}[(a)]
 \item
 the $\eps$-neighborhoods   of all points $-a^+_j, -b^+_j,  -\be^+_k, 
 -\al^+_k$ do not intersect;
 \item
 for all $\be\in \cup_j[-b^+_j,-a^+_j]$, $\Im F(\be-i\eps)\le -c$;
 
 \item
 for all $-b^+_j$ and  $\varphi \in (-\pi/2,\pi/2)$, 
 $F(-b^+_j-\eps e^{i\varphi})= -C_j\ln \eps+O(1)$, where $C_j>0$;
  \item
 for all $-a^+_j$ and  $\varphi \in (-\pi/2,\pi/2)$, 
 $F(-a^+_j+\eps e^{i\varphi})= C'_j\ln \eps+O(1)$, where $C'_j>0$.
  \end{enumerate}
It follows that if $\eps>0$ is sufficiently small,  then
\begin{enumerate}[(i)]
\item
 for each $j$ and $z\in [-b^+_j,-a^+_j]-i\eps$, $\mathrm{arg}\, F(z)\in (-\pi,0)$;
 \item
 on each connected interval of $\cL_{\eps,0}$, $\mathrm{arg}\, F(z)$ equals either $\pi$ or $-\pi$;
 \item
 for all $j$, $\mathrm{arg}\, F(-a^+_j+i\eps)=-\pi,$ $\mathrm{arg}\, F(-a^+_j-i\eps)\in (-\pi,0)$,
 $\mathrm{arg}\, F(-b^+_j-i\eps)=0,$ $\mathrm{arg}\, F(-b^+_j-i\eps)\in (-\pi,0)$,
 \item
 as $z$ moves down from $-a^+_j+\eps$ to $-a^+_j-i\eps$, $w:=F(z)\in \{\Re w<0\}$ and runs from a point on $(-\infty,0)$ to a point 
 in the quadrant $\{\Re w<0, \Im w<0\}$;
 \item
 as $z$ moves down from $-b^+_j-i\eps$ to $-b^+_j-\eps$, $w:=F(z)\in\{\Re w> 0\}$ and runs from a point on $(0,+\infty)$ to a point 
 in  the quadrant  $\{\Re w>0, \Im w<0\}$;
  \item
 as $z$ moves down along each semi-circle, $F(z)$ 
 remains in $\{\Re w\le 0, w\neq 0\}$ and $\mathrm{arg}\, F(z)$  varies either from $-\pi$ to 0 or
 from 0 to $-\pi$ or from 0 to 0, or from $-\pi$ to $-\pi$. 
 \end{enumerate}
 It follows that as $z$ runs from $-i\eps$ to $-1/\eps$, $F(z)\not\in\{\Im w>0\}\cup\{0\}$.
 Hence, $\mathrm{ind}\, (\cL_\eps, -i\eps,-1/\eps; F)=(-\pi-0)/(2\pi)=-1/2$. Using $\overline{F(z)}=F(\bar z)$ and the symmetry $z\mapsto -z$, we observe that
if we move along $\cL_\eps$ in the opposite direction (denote $\cL_\eps$ with the opposite direction by $\cL'_\eps$), then 
\[
\frac{1}{2\pi}\int_{\cL'_\eps(0,1/\eps)}d\,\mathrm{arg} F(z)=1/2.
\]
Hence, $\mathrm{ind}(\cL_\eps, 1/\eps, 0; F)=-1/2$, and 
\[
\mathrm{ind}(\cL, 0,0; F)=\mathrm{ind}(\cL_\eps, 0, -1/\eps; F)+\mathrm{ind}(\cL_\eps, -1/\eps, 1/\eps; F)
+\mathrm{ind}(\cL_\eps, 1/\eps, 0; F)=0.
\]
  This finishes the proof of Theorem \ref{thm:zeroF}.

  \section{Mixing and subordination of SINH-, sSL- and SL-processes}\label{s:mixing_sub}
    \subsection{Mixing}
  \subsubsection{Mixing SINH-regular processes}\label{ss:mixingSINH}
  The following theorem is immediate from
the definition of SINH-regular processes. 
\begin{thm}\label{fin_mix_SINH}
Let the following conditions hold
\begin{enumerate}[(i)]
\item
processes $X^j$, $j=1,2\ldots, N$, are independent;
\item
each $X^j$  is a SINH-regular L\'evy process  of order $(\nu'_j,\nu_j)$ and type 
$((\mum^j, \mup^j), \cC^j, \cC_+^j)$;
\item
$\mum:=\max_j\mum^j<\mup:=\min_j\mup^j$ and $\cC:=\cap_{j=1}^N \cC^j\ne \emptyset$;
\item
there exists $\nu'\in (\min_j \nu'_j, \max_j\nu_j)$ such that $\cC_+:=\cap_{\nu'_j\ge \nu'}\cC^j_+\ne \emptyset$;

\end{enumerate}
Then, for any $a_j>0, j=1,2,\ldots, N$, $X:=\sum_{j=1}^N a_j X^j$ is SINH-regular of order $(\nu',\max_j\nu_j)$ and type
$((\mum,\mup), \cC,\cC_+)$.
\end{thm}
For the sake of brevity, we omit  straightforward (albeit messier) generalizations to the case of infinite and integral mixing.

\subsubsection{Mixing  sSL- and SL-processes}\label{ss:mixing_sSL_SL}
sSL- and SL-regular processes are mixtures of the BM and pure jump
spectrally one-sided processes, hence, it suffices to consider mixtures of spectrally positive processes
(results for spectrally negative  processes are by symmetry).

Since the sum of Stiltjes measures (and, under additional regularity conditions, integral of a family of Stieltjes measures
$\cG_\al, \al\in \cA$)
is a Stieltjes measure, the
mixture of processes $X^\al\in sSL^{+}_{\mu_\al}$ (resp., $X^\al\in sSL^{2;+}_{\mu_\al}$; $X^\al\in sSL^{1;-}_{\mu_\al}$), where $\mu_\al\ge \mu>0, \al\in \cA,$ is a process of class $sSL^{+}_{\mu}$ (resp., $sSL^{2;-}_{\mu}$; $sSL^{1;+}_{\mu}$), and the same is true for SL-processes. If $\inf\mu_\al=0$, then sufficient conditions for the mixture to be of the same class are more involved. We omit this case for the sake of brevity.

  \subsubsection{Mixing regular sSL- and SL-processes}\label{ss:mixing_regSL} 
  Conditions in Sections \ref{ss:mixingSINH} and \ref{ss:mixing_sSL_SL} must hold.
  
  \subsection{Subordination}
  A L\'evy process taking values in $[0, +\infty)$, 
which implies that its trajectories are increasing a.s., is called a subordinator. Since  non-decreasing paths have bounded variation, the characteristic exponent of a subordinator is of the
form
\begin{equation}\label{subpsi}
\psi(\xi)=-ib \xi
+\int_0^{+\infty}(1-e^{ix\xi})F(dx),
\end{equation}
where $b\ge 0$.
The Laplace transform of the law of a subordinator $Z$ can be
expressed as
\begin{equation}\label{sublap}
\bE[\exp(-q Z_t)]=\exp(-t\Psi(q)),
\end{equation}
where $\Psi:\bR_+\to\bR_+$ is called the {\em Laplace exponent}
of $Z$.\footnote{This definition is marginally different from the definition \cite[Eq. (30.1)]{sato}. We change the definition
in accordance to the change of the definition of the characteristic exponent that we use so that the form of the key theorem (Thm. \ref{thsub}) does not change.} Thus, the characteristic exponent of $Z$ is $\psi(\xi)=\Psi(-i\xi)$, and
\begin{equation}\label{subPhi}
\Psi(q)=\psi(iq)=b q+\int_0^{+\infty}(1-e^{-q x})F(dx).
\end{equation}
If $X$ is a subordinator, $\Psi(q)$ admits analytic continuation  to the right half-plane
$\{\Re q>0\}$, and $\psi(\xi)$ to the upper half-plane $\{\Im\xi>0\}$. If $Z$ is SINH-regular of type $((\mum,+\infty), i\cC_\ga, i\cC_{\ga'})$, where $\mum\le 0$ and $\pi/2<\ga'<\ga\le\pi$, $\Psi$ admits analytic continuation to $(\mum,+\infty)+(\cC_\ga\cup\{0\})$. 
\begin{lem}\label{lem:BernSub} 
Let 
$Z^+\in SL^{1,+}_\mu$, $\mu\ge 0$, be a subordinator.
 Then, for any $\ga\in [0, \pi)$, $\Psi_Z(\mu+\cC^+_\ga)\subset \mu+(\cC^+_\ga\cup\{0\})$, where $\cC^+_\ga=\{z\ |\ \mathrm{arg}\, z\in [0,\ga]\}$.
 \end{lem}
 \begin{proof} The Laplace exponent of an
SL-subordinator is a complete Bernstein function (see \cite[Defin.6.1]{schilling_book_Bernstein2012} for the definition of the latter). 
Hence, the statement of Lemma is immediate from \cite[Cor. 6.6]{schilling_book_Bernstein2012}
(formulated for the case $b=0$). 
 \end{proof}

\begin{thm}(\cite[Thm 30.1]{sato})\label{thsub}
Let $Z$ be a subordinator with the Laplace exponent $\Psi$, let
$Y$ be a L\'evy process with the characteristic exponent $\kappa$,
and suppose that $Z$ and $Y$ are independent.

Define $X_t(\om)=Y_{Z_t(\om)}(\om),\ t\ge 0$. Then $\{X_t\}=\{Y_{Z_t}\}$ is a L\'evy
process with the characteristic exponent $\psi(\xi)=\Psi(\kappa(\xi))$.
\end{thm}

\subsubsection{Subordination of SINH-regular processes}  For the sake of brevity, we
restrict ourselves to the case
of SINH-regular processes of order $\nu\in (0,2]$, with the cones $\cC_+$ containing $\bR\setminus 0$.

\begin{thm}\label{thm:subSINH} Let the following conditions hold
\begin{enumerate}[(i)]
\item $Z$ is a SINH-regular subordinator of order $\nu_Z\in (0,1)\cup\{0+\}
$ 
and type $((\lm,+\infty), i\cC_\ga, i\cC_{\ga'})$,
where $\lm<0$, $\pi/2<\ga'<\ga\le \pi$, with the drift $\mu_Z\ge 0$;
\item
$Y$ is a SINH-regular L\'evy process of order $\nu_Y\in (0,2]$ and type $((\mum,\mup), \cC, \cC_+)$, 
where $\mum<0<\mup$, and $\cC_+\supset \bR\setminus 0$,
with the characteristic exponent $\psi_Y(\xi)=-i\mu_Y\xi+\psi^0_Y(\xi)$;
 \item
 $Z$ and $Y$ are independent.
 \end{enumerate}
 Then 
 \begin{enumerate}[(1)]
 \item
 there exist $\tilde\mum\in (\mum,0), \tilde\mup\in (0,\mup)$ and open coni $\tilde\cC\subset\cC$, 
$\tilde\cC_+\subset\cC_+$ such that $\bR\setminus\{0\}\subset\tilde\cC_+\subset\tilde\cC$,  and 
\beqa\label{eq:region_of_anal}
\psi_Y(i(\tilde\mum,\tilde\muppr)+(\tilde\cC\cup\{0\}))&\subset &(\lm,+\infty)+(\cC_\ga\cup\{0\}),\\\label{eq:region_of_anal+}
\psi_Y(i(\tilde\mum,\tilde\muppr)+(\tilde\cC_+\cup\{0\}))&\subset &(\lm,+\infty)+(\cC_{\ga'}\cup\{0\}).
\eqa
 \item
 $\{X_t\}=\{Y_{Z_t}\}$ is SINH-regular of type $((\tilde\mum,\tilde\mup), \tilde\cC, \tilde\cC_+)$, and the
 order $\nu_X$, where
 \begin{enumerate}[(a)]
 \item if $\mu_Z>0$, $\nu_X=\nu_Y$;
 \item if $\mu_Z=0$, $\nu_Z\in (0,1)$, and either $\nu_Y\in [1,2]$ or $\nu_Y\in (0,1)$ and $\mu_Y=0$,
 then $\nu_X=\nu_Z\nu_Y$;
 \item
 if $\mu_Z=0$, $\nu_Z\in (0,1)$, and $\nu_Y\in (0,1)$, $\mu_Y\neq 0$, then $\nu_X=\nu_Z$;
 \item
 if $\mu_Z=0$, $\nu_Z=0+$, then $\nu_X=0+$.
  \end{enumerate}
   \end{enumerate}
 \end{thm}
 \begin{proof} (1). It suffices to note that, for $\xi\in\bR$, $|e^{-t\psi(\xi)}|=|\bE\left[e^{i\xi Y_t}\right]|\le 1$, hence, $\Re\psi(\xi)\ge 0$;
 and  $\psi(\xi)$ (resp., $\Re\psi(\xi)$) stabilizes to a positively homogeneous function as $\xi\to\infty$ remaining in $\cC$
 (resp.,  $\cC_+$). 
 
 (2).
 Let $\varphi\in (-\pi/2,\pi/2)$ and $\rho\to+\infty$. Set $\xi=\rho e^{i\varphi}$ and consider
 \beqast
 \psi_X(\xi)&=&\mu_Z(-i\mu_Y\xi+\psi^0_Y(\xi))+\Psi^0_Z(-i\mu_Y\xi+\psi^0_Y(\xi))\\
 &=&-i\mu_Z\mu_Y\xi+\mu_Z\psi^0_Y(\xi)+\psi^0_Z(\mu_Y\xi+i\psi^0_Y(\xi)).
 \eqast
 If $\mu_Z\neq 0$, then the third term increases slower than the second one, and (a) follows.
 If the condition in (b) is satisfied, then 
 $
 \psi^0_Z(\mu_Y\xi+i\psi^0_Y(\xi))\sim \psi^0_Z(ic_{Y,\infty}(\varphi)\rho^{\nu_Y}),
 $
 and since $\psi^0_Z(\xi)$ stabilizes to a positively homogeneous function of degree $\nu_X$ as $\xi\to\infty$,
 the conclusion in (b) is proved.
 If the condition in (c) is satisfied, $\psi^0_Z(\mu_Y\xi+i\psi^0_Y(\xi))\sim \psi^0_Z(\mu_Y\xi)$, and (c) foolows.

 Finally, if $\nu_Z=0+$, then we note that $\ln (c_{Y,\infty}(\varphi)\rho^\nu)\sim \nu_Y\ln\rho$, which proves (d).
 
 \end{proof}

\subsubsection{Subordination of and by sSL- and SL-processes}
Let $Y\in sSL_{\mum,\mup}$, $\mum\le 0\le \mup$, $\mum< \mup$, and let $Z\in sSL^+_\mu, \mu> 0,$ be a subordinator with the L\'evy exponent $\Psi_Z$.
Assume that the characteristic function $\psi_Y$ enjoys the property \eq{eq:no_zeros}.
By Theorem \ref{thm:zeros_SL} (a),  if $Y$  is a SL-process, \eq{eq:no_zeros} is satisfied.

Notice that the following possibilities exist:
\begin{enumerate}[(1)]
\item
$\psi_Y(i(\mum+0))\ge  -\mu$ and $\psi_Y(i(\mup-0))\ge  -\mu$;
\item
$\psi_Y(i(\mum+0))< -\mu$ and $\psi_Y(i(\mup-0))< -\mu$;
\item
$\psi_Y(i(\mum+0))\ge  -\mu$ and $\psi_Y(i(\mup-0))<  -\mu$;
\item
$\psi_Y(i(\mum+0))< -\mu$ and $\psi_Y(i(\mup-0))\ge  -\mu$.
\end{enumerate}
\begin{lem}\label{lem:mumpr}
\begin{enumerate}[(a)]
\item
If $\psi_Y(i(\mum+0))< -\mu$, there exists $\mumpr\in (\mum,\mup)$ such that $\psi_Y(i\mumpr)=-\mu$,
and $t\mapsto \psi_Y(it)$ increases on $(\mum,\mumpr]$.
\item
If $\psi_Y(i(\mup-0))< -\mu$, there exists $\muppr\in (\mum,\mup)$ such that $\psi_Y(i\muppr)=-\mu$,
and $t\mapsto \psi_Y(it)$ decreases on $[\muppr,\mup)$.
\item
If \eq{eq:no_zeros} holds, $\psi_X(\xi)=\Psi_Z(\psi_Y(\xi))$ is analytic in $\bC\setminus i((-\infty,\mu^X_-]\cup [\mu^X_+,+\infty))$, where\\ in Case (1), $\mu^X_-=\mum$, $\mu^X_+=\mup$; in Case (2), $\mu^X_-=\mumpr$, 
$\mu^X_+=\muppr$;\\ in Case (3), $\mu^X_-=\mum$, $\mu^X_+=\muppr$; in Case (4), $\mu^X_-=\mumpr$,
$\mu^X_+=\mup$.
\end{enumerate}
\end{lem}
\begin{proof} It suffices to note that  $(\mum\mup)\ni t\mapsto \psi_Y(it)\in \bR$ is concave, and that under condition \eq{eq:no_zeros},
 $\psi_Y(\xi)\in (-\infty,-\mu]$ for $\xi\in \bC\setminus i((-\infty,\mum]\cup [\mup,+\infty))$ implies $\xi\in i(\mum,\mup)$. 
\end{proof}
\begin{thm}\label{thm:sub_sSL_SL}
Let the following conditions hold:
\begin{enumerate}[(i)]
\item $Y\in sSL_{\mum,\mup}$, and $\psi_Y$ satisfy \eq{upbound} and \eq{eq:no_zeros};  
\item
$Z\in sSL^+_\mu$, and there exist
$\nu,\de\in (0,1]$ and $C>0$ s.t. 
\bbe\label{upboundPsi}
|\Psi_Z(z)|\le C(|z|^\nu+|z-\mu|^{-1+\de}), \ z\in \bC\setminus (-\infty,-\mu);
\ee
\item
the limits $\psi_Y(it+0), t\in (-\infty,\mum)\cup (\mup,+\infty)$ and $\Psi_Z(-t+i0), t>\mu$, exist;
\item
 $Z$ and $Y$ are independent.
\end{enumerate}
Then 
\begin{enumerate}[(a)]
\item
$\{X_t\}=\{Y_{Z_t}\}\in sSL_{\mu^X_-,\mu^X_+}$, and the sSL-measures of $X$ are absolutely continuous: $\cG_{X;\pm}(dt)=g_{X;\pm}(t)dt \in \sSLM_{\mp\mu_\mp}$, where 
\beqa\label{eq:cGp}
g_{X;+}(t)&=&\frac{1}{\pi}\times\begin{cases} \Psi_Z(\psi_Y(-it-0)), & t\in (-\mum,+\infty),\\
\Psi_Z(\psi_Y(-it)-i0), & t\in (-\mu^X_-,-\mum),\\
0, & t\in (-\infty,-\mu^X_-),
\end{cases}
\\\label{eq:cGm}
g_{X;-}(t)&=&\frac{1}{\pi}\times\begin{cases} \Psi_Z(\psi_Y(it+0)), & t\in (\mup,+\infty),\\
\Psi_Z(\psi_Y(it)+i0), & t\in (\mu^X_+,\mup),\\
0, & t\in (-\infty,\mu^X_+);
\end{cases}
\eqa
\item
if, in addition, $Y$ and $Z$ are SL-processes, then $g_{X;\pm}\ge 0$, and $X$ is an SL-process.
\end{enumerate}
\end{thm}
\begin{proof} (a) is immediate from Theorem \ref{thm:SLverification} and Lemma \ref{lem:mumpr}.
At $t\in (-\mu^X_-,-\mum)$, we use the linearization 
$
\psi_Y(-it-\eps)\sim\psi_Y(-it)+i\eps\frac{d\psi_Y(-it)}{dt}(t)$ and the inequalities $\psi_Y(-it)<\mu$, $d\psi_Y(-it)/dt<0$.
The case  $t\in (\mu^X_+,\mup)$ is by symmetry.

 (b) follows from 
Theorem \ref{thm:SLverification} and Lemma \ref{lem:BernSub}. 

\end{proof}

\subsection{Construction of SL- and sSL-processes as SL- and sSL-subordinated Brownian motion}\label{ss:constr_SL_as_SLsub_BM} First, we consider the case of absolutely continuous sSL and SL measures, and then the case
of discrete measures. Below, $\sqrt{\cdot} $ denotes the standard branch of the square root, and $W$ is the 
standard Wiener process.

\begin{thm}\label{thm:SLsubBM}
Let the following conditions hold:
\begin{enumerate}[(i)]
\item
$X\in SL_{\mum,\mup}$, where $\mum\le 0\le \mup$, $\mum<\mup$;
\item
\bbe\label{eq:sym}
\psi_X(\xi-i\be)=\psi(-\xi+i\be), \ \forall \xi\in \bC\setminus i((-\infty,\mum]\cup [\mup,+\infty)),
\ee  
where $\be=(\mum+\mup)/2$;
\item
\bbe\label{eq:limA}
\lim_{A\to+\infty}\int_{|\xi|=A, \xi\not\in i\bR}\frac{|\psi_J(\xi)|}{|\xi|^3}d\xi=0.
\ee
\end{enumerate} 

Then
\begin{enumerate}[(a)]
\item
function $\Psi_Z(q):=\psi(\sqrt{q}-i\be)-\psi(-i\be)$ is the Laplace exponent of a subordinator $Z$,
under the Esscher transform $\bQ_\be$ which makes the characteristic exponent of $X$ symmetric:
$\psi_{\bQ_\be}(\xi)=\psi_{\bQ_\be}(-\xi)$;
\item
$Z\in SL^+_\mu$, where $\mu=(\mup-\mum)^2/4$;
\item
the Stieltjes measure of the complete Bernstein function $\Psi_Z$ is given by
\bbe\label{StM_Psi_Z}
\cG^0_Z((u,v])=\frac{1}{\pi}\lim_{\eps\to 0+}\int_{(u,v]}\frac{\Im\psi_X(i(\sqrt{q}-\be)+\eps)}{q}dq,
\ee
for any $u>\be^2$ and $v>u$ s.t. $\sqrt{u}-\be$ and $\sqrt{v}-\be$ are points of continuity of the distribution function
$t\mapsto \cG_X((-\infty,t])$;
\item
 $X_t=Y_{Z_t}$, where, under $\bQ_\be$, $Y$ is given by $dY=\sqrt{2}dW$. 
\end{enumerate}
\end{thm}
\begin{proof} 
In view of the symmetry condition \eq{eq:sym}, we can use the Esscher transform to reduce the proof to the case of
the symmetric $\psi=\psi_X$: $\psi(\xi)=\psi(-\xi)$ and $\be=0$. Then $\mup=-\mum$, $\mu=\mu_\pm^2$,
$\Psi_Z$ admits analytic continuation to $\bC\setminus (-\infty,-\mu]$, and
$\Psi_Z(\xi^2)=\psi(\xi)$ for all $\xi\in \bC\setminus i((-\infty,\mum])\cup[\mup,+\infty))$. Hence, (d) follows from (a).

To prove (a)-(b), first, we verify that $\cG^0_Z$ given by \eq{StM_Psi_Z} is the Stieltjes measure:
\beqast
\int_{(\mu,+\infty)}\frac{\cG^0_Z(dq)}{1+q}&=&\frac{1}{\pi}\lim_{\eps\to 0+}\int_\mu^{+\infty}
\frac{\Im\psi_X(i\sqrt{q}+\eps)}{q(1+q)}dq\\
&=&\frac{2}{\pi}\lim_{\eps\to 0+}\int_{0}^{+\infty}
\frac{\Im\psi_X(it+\eps)dt}{t(1+t^2)}
=2\int_{(0,+\infty)}\frac{\cG_X(dt)}{t+t^3}.
\eqast
The last equality follows from \eq{reprcGp}; the integral is finite since $\cG_X$ is an SL-measure.

Due to the symmetry condition $\psi(\xi)=\psi(-\xi)$, 
$\psi$ is of the form $\psi(\xi)=\frac{\sg^2}{2}\xi^2+\psi_J(\xi)$, where $\psi_J(\xi)(=\psi_J(-\xi))$
is the pure jump component. Hence, the drift of $Z$ is $\sg^2/2$, and it remains to prove that
$\psi_J(\xi)=\xi^2 ST(\cG^0_Z)(\xi^2)$. In \eq{StM_Psi_Z}, we can replace $\psi_X$ with $\psi_J$, and obtain, for any 
$\mu_1\in (0,\mu)$ and $\xi^2\not\in (-\infty,-\mu]$,
\[
\xi^2 ST(\cG^0_Z)(\xi^2)=\frac{1}{\pi}\lim_{\eps\to 0+}\int_{(\mu_1,+\infty)}\frac{\xi^2}{\xi^2+q}\frac{\Im\psi_J(i\sqrt{q}+\eps)}{q}dq.
\]
We change the variable $q=t^2$, and, for any $\mu_2\in (0,\mup)$, obtain
\bbe\label{psiJZ}
\xi^2 ST(\cG^0_Z)(\xi^2)=\frac{2}{\pi}\lim_{\eps\to 0+}\int_{(\mu_2,+\infty)}\frac{\xi^2}{\xi^2+t^2}\frac{\Im\psi_J(it+\eps)}{t}dt.
\ee
Let $\cL_\eps=\{z\in \bC\ |\ \mathrm{dist}\, (z, (\mup,+\infty))=\eps\}$. Since $\Im\psi_J(it+\eps)=(\psi_J(it+\eps)-\psi_J(it-\eps))/2i$, we 
may rewrite \eq{psiJZ} as 
\bbe\label{psiJZ2}
\xi^2 ST(\cG^0_Z)(\xi^2)=\frac{1}{\pi i}\lim_{\eps\to 0+}\int_{\cL_\eps}\frac{\xi^2}{\xi^2+z^2}\frac{\psi_J(iz)}{z}dz,
\ee
where $\cL_\eps$ is passed from $+\infty+i\eps$ to $+\infty-i\eps$. Let the contour $-\cL_\eps$ be passed from
$-\infty-i\eps$ to $-\infty+i\eps$. Then, on the strength of 
\eq{eq:limA},  \[
\frac{1}{2\pi i}\left(\int_{\cL_\eps}+\int_{-\cL_\eps}\right)\frac{\xi^2}{\xi^2+z^2}\frac{\psi_J(iz)}{z}dz\]
plus the sum of residues of the integrand equals 0. Since $\psi(-iz)=\psi(iz)$, the integrals over $\cL_\eps$ and $-\cL_\eps$ are equal,
hence, the RHS of \eq{psiJZ2} is opposite to the sum of residues. The apparent singularity at 0 is removable
and $\psi_J(\xi)=\psi_J(-\xi)$, hence, the sum of residues is
\[
\frac{\xi^2\psi_J(-\xi)}{(2i\xi)(i\xi)}+\frac{\xi^2\psi_J(\xi)}{(-2i\xi)(-i\xi)}=-\psi_J(\xi).
\]
Thus, 
$\psi_J(\xi)=\xi^2 ST(\cG^0_Z)(\xi^2)$, which finishes the proof.
\end{proof}

\begin{thm}\label{thm:sSLsubBM}
Let the following conditions hold:
\begin{enumerate}[(i)]
\item
$X\in sSL_{\mum,\mup}$, where $\mum\le 0\le \mup$, $\mum<\mup$, satisfies \eq{eq:sym};
\item
conditions of Theorem \ref{thm:SLverification} are satisfied;
\item
the Laplace transform of the measure $\cG(dq)=\pi^{-1}\Im\psi_X(i(\sqrt{q}-\be)-0))dq$ is positive.
\end{enumerate}
 Then
\begin{enumerate}[(a)]
\item
$\cG(dq)$ is an sSL-measure;
\item
function $\Psi_Z(q):=\psi(\sqrt{q}-i\be)-\psi(-i\be)$ is the Laplace exponent of a subordinator $Z$;
\item
$Z\in sSL^+_\mu$, where $\mu=(\mup-\mum)^2/4$, and $\cG(dq)$ is the sSL-measure of $Z$;
\item
$\cG(dq)$ is the sSL-measure of $Z$;
\item
 $X_t=Y_{Z_t}$, where $Y$ is given by $dY=\sqrt{2}dW$  as in Theorem \ref{thm:SLsubBM}.
\end{enumerate}
\end{thm}
\begin{proof} 
Theorem \ref{thm:SLverification} allows us to prove (a); the only subtlety, namely,
the proof of the condition (iii), is assumed away. The rest of the proof is the same as in the case of $X\in SL_{\mum,\mup}$.

\end{proof}
\begin{example}\label{ex:subMeixner}
{\rm
Let $X$ be a Meixner process. Conditions (i)-(ii) are verified in Example \ref{ex:Meixnerdirect0}.
In view of \eq{gp_Meixner}, it suffices to prove that, for $0<u<v$ and $x>0$,
\bbe\label{eq:posLapl}
\left(\int_{\sqrt{u}}^{\sqrt{v+u}}-\int_{\sqrt{u+v}}^{\sqrt{u+2v}}\right) e^{-q x}\cG(dq)>0.
\ee
Since $x>0$, it suffices to prove that $2\sqrt{v+u}-\sqrt{u}>\sqrt{u+2v}$; this inequality holds since the function $u\mapsto\sqrt{u}$ is
concave.

}
\end{example}

Now we consider SL processes with discrete SL-measures. Assume that
the measure is symmetric; 
using the Esscher transform, one can generalize the result below.
\begin{thm}\label{thm:sub_BM-disc}
Let $X$ be a driftless SL-process with the symmetric discrete L\'evy measure; equivalently,
\[
\cG_\pm=\sum_{\al\in \cA} p_\al \la_\al\de_{\la_\al},
\]
where $\cA$ is a countable set, $p_\al, \la_\al>0$, and there exist $a_1,a_2>0$ such that 
\bbe\label{eq:SL_disc_meas}
\sum_{\al\in \cA}\frac{p_\al}{(a_1+a_2\la_\al)(1+\la_\al)}<\infty.
\ee
Then $X_t=Y_{Z_t}$, where  $Z$ is an SL-subordinator with the SL measure
$\cG_Z=2\sum_{\al\in\cA} p_\al \la_\al^2\de_{\la_\al}$, and $Y=2W.$
\end{thm}
\begin{proof} For the sake of brevity, assume that $X$ has no BM component. Then
\[
\psi(\xi)=\sum_{\al\in\cA}p_\al\left(\frac{-i\xi}{\la_\al-i\xi}+\frac{i\xi}{\la_\al+i\xi}\right)=2\sum_{\al\in\cA}p_\al\frac{\xi^2}{\la_\al^2+\xi^2}.
\]
Hence, $X_t=Y_{Z_t}$, where the Laplace exponent of $Z$ is given by 
\[
\Psi_Z(q)=2\sum_{\al\in\cA}p_\al\frac{q}{\la_\al^2+q},
\]
provided we prove that $\cG_Z$ is an SL-measure, equivalently, 
\[
\sum_{\al\in\cA}p_\al\frac{\la_\al^2}{\la_\al^2(1+\la_\al^2)}<\infty.
\]
The last inequality follows from \eq{eq:SL_disc_meas}.

\end{proof}

 \section{Conclusion}\label{concl}
The conformal deformations technique (which is similar to but simpler and more
flexible than the saddle point method) allows one to develop
efficient numerical methods for the evaluation of the Wiener-Hopf factors
and various probability distributions (prices of options of several
types) in L\'evy models and models with (conditional) infinitely
divisible distributions. The crucial conditions are: 1) the
characteristic exponent $\psi$ admits analytic continuation to a
union $U$ of a strip and cone around or adjacent to $\bR$; 2)
$\Re\psi(\xi)\to+\infty$ as $(U\ni)\xi\to \infty$.  We call processes
satisfying 1)-2) SINH-regular because the most efficient family of
conformal deformations is in terms of the function $\sinh$.  We
showed that essentially all popular classes of L\'evy processes other
than stable L\'evy processes are SINH-regular, and calculated the
corresponding strips and coni\footnote{In the case of stable L\'evy processes,
the strip degenerates into $\bR$ but a modified conformal
deformation technique is applicable in this case as well
\cite{ConfAccelerationStable}.}. Choices of appropriate conformal
deformations simplify if the cone of analyticity is the set
$\bC\setminus i\bR$. We constructed a class of signed
Stieltjes-L\'evy (sSL-) processes enjoying this property, and showed
that all popular classes of L\'evy processes except for the Merton model are
sSL-processes. The construction is based on the representation of
L\'evy densities of positive and negative jumps as Laplace transforms
of measures of the form $(a_2t^2+a_1t)\cG_\pm(dt)$, where
$\cG_\pm(dt)$ are the differences of Stieltjes measures. If
$\cG_\pm(dt)$ are Stieltjes measures, then we say that the process is
a Stieltjes-L\'evy (SL-) process.  We proved that SL-processes enjoy
an additional property, namely, the absence of solutions of the
equation $\psi(\xi)+q=0$ for any $q>0$, which simplifies
calculations of the Wiener-Hopf factors, calculations of joint
probability distributions of a L\'evy process and its extrema, and
pricing options with barrier and/or lookback features.  We proved
that all popular classes of L\'evy processes except for the Merton
model and the Meixner processes are SL-processes, and that Meixner processes are
sSL-processes. We derived a representation of $\psi$ in terms of
the Stieltjes functions, and a representation of the measures
$\cG_\pm(dt)$ in terms of $\psi$. One of the representation theorems
contains a set of sufficient conditions on a function $\psi$ to be a
characteristic exponent of a L\'evy process $X$, and conditions for
$X$ to be an sSL-process or an SL-process. We showed that the theorem is
applicable to all popular L\'evy processes, and used the theorem to
prove that, under a natural symmetry condition, all popular classes
of processes 
are subordinated Brownian motions. We proved that, under additional weak
regularity conditions, mixtures of SINH-regular processes,
sSL-processes and SL-processes are processes of the same class, and
that the subordination of an SL-process by an sSL- (resp., SL-)
subordinator is an sSL- (resp., SL-) process.

We leave to the future the study of the further generalization
(generalized SL-processes: gSL-processes), which is needed to include
in the general framework L\'evy processes with L\'evy densities given
by exponential polynomials. The generalization can be obtained
allowing for derivatives of measures (in the sense of generalized
functions).
 
The following extensions of the results of the paper seem to be of interest:

1) describe classes of infinitely divisible distributions (e.g.,
conditional distributions on stochastic volatility models and models
with stochastic interest rates) and additive processes that lead to
$\psi$ of sSL- or SL-processes;

2) using the representation in terms of Stieltjes measures, study
efficient approximations of processes with absolutely continuous
measures by processes with discrete measures (e.g., approximations of
KoBoL by HEJD model or meromorphic processes).  Examples in the paper
suggest that approximations by measures with uniformly spaced atoms
could be less efficient than approximations by non-uniformly spaced
measures;

3) study the relative efficiency of Monte-Carlo procedures for sSL- and
SL-processes based on the subordinated BM-representation as in
\cite{madan-yor} vs the approximation of the transition probability
density using the conformal deformation technique
\cite{MCMityaLevy,SINHregular,ConfAccelerationStable}. Numerical
examples in \cite{MCMityaLevy} suggest that the latter is more
efficient than the former;

4) generalizations for L\'evy processes on $\bR^n$.

 \bibliography{StronglyRegular5.bbl}{}
\bibliographystyle{plain}

\appendix
\section{Technicalities}\label{tech}

\subsection{The order and type of KoBoL processes}\label{ss:KoBoL}
\subsubsection{KoBoL processes of order $\nu\in (0,2), \nu\neq 1$}\label{OnesidedKoBoLnu02}  
\begin{prop}\label{ell_KBL_neq01} Let $\psi^0$ be given by \eq{KBLnupnumneq01}. Then 
\begin{enumerate}[(i)]
\item 
Spectrally positive KoBoL of order $\nu_+\in(0,2)$, $\nu_+\neq 1$, are SINH-regular of order $\nu_+$ and type 
$((\lm,+\infty), \bC\setminus i(-\infty,0], \cC_+)$, where 
\begin{enumerate}[(1)]
\item  if $\nu_+\in (1,2)$,
$\cC_+=\cC_{\gampr,\gappr}$, $\gampr=\max\{-1,(1-3/\nu_+)\}\pi/2$, $\gappr= (1-1/\nu_+)\pi/2$;
\item
 if $\nu_+\in (0,1)$, $\cC_+=i\cC_{\min\{1,1/(2\nu_+)\}\pi}$. 
\end{enumerate}
We have
\bbe\label{ascofnup}
c_\infty(\varphi)=-c_+\Ga(-\nu_+)\exp[i(-\pi/2+\varphi)\nu_+].
\ee
\item
If $\nu_+>\nu_-$ and $c_\pm >0$, then $X$ is
 SINH-regular of order
$\nu_+$ and type $((\lm,\lp), \bC\setminus i\bR, \cC_+)$, where $c_\infty$ is given by \eq{ascofnup}, and
$\cC_+$ is the same as in (i).
\item
Spectrally negative KoBoL of order $\nu_-\in(0,2)$, $\nu_-\neq 1$, are  SINH-regular of order $\nu_-$ and type 
$((-\infty,\lp), \bC\setminus i[0,+\infty), \cC_+)$, where \begin{enumerate}[(1)]
\item if $\nu_-\in (1,2)$,
$\cC_+=\cC_{\gampr,\gappr}$,  $\gampr=(1/\nu_--1)\pi/2$, $\gappr=\min\{1,3/\nu_--1\}\pi/2$, and 
\item
if $\nu_-\in (0,1)$, $\cC_+=-i\cC_{\min\{1,1/(2\nu_-)\}\pi}$. 
\end{enumerate}
We have
\bbe\label{ascofnum}
c_\infty(\varphi)=-c_-\Ga(-\nu_-)\exp[i(\pi/2+\varphi)\nu_-].
\ee
\item
If $\nu_->\nu_+$ and $c_\pm>0$, then $X$ is
 SINH-regular of order
$\nu_-$ and type $((\lm,\lp), \bC\setminus i\bR, \cC_+)$, where $c_\infty$ is given by \eq{ascofnup}, and
$\cC_+$ is the same as in (iii).
\item
If $\nu_+=\nu_-=\nu$ and $c_\pm\neq 0$, then $X$ is
 SINH-regular of order
$\nu_-$ and type $((\lm,\lp), \bC\setminus i\bR, \cC_+)$, where $\cC_+=\cC_{\gam,\gap}$,
$(\gam,\gap)=\{\varphi\in (-\pi/2,\pi/2)\ |\ c_\infty(\varphi)>0\}$, and
\bbe\label{ascofnupeqnum}
c_\infty(\varphi)=-\Ga(-\nu)[c_+e^{-i\pi\nu/2}+c_-e^{i\pi\nu/2}]e^{i\varphi\nu}.
\ee
\item
In particular, if $\psi^0$ is given by \eq{KBLnuneq01}, then $\gap=\ga_\nu:=\min\{1,1/\nu\}\pi/2$, $\gam=-\gap$,
and \eq{ascofnupeqnumcc} holds.

\end{enumerate}

\end{prop}
\begin{proof}
Evidently, $\psi^0$ is analytic in the complex plane with the cuts $i(-\infty,\lm]$ and $i[\lp,+\infty)$,
hence, $\mu_\pm=\la_\pm$ and $\cC=\bC\setminus i\bR$.
Consider the asymptotics of the characteristic exponents of one-sided KoBoL processes
\beqa\label{onesidedKBLp}
\psi^0_+(\nu_+, c_+, \lm ;\xi)&=&c_+\Ga(-\nu_+)((-\lm)^{\nu_+}-(-\lm-i\xi)^{\nu_+}),\\
\label{onesidedKBLm}
\psi^0_-(\nu_-, c_-, \lp ;\xi)&=&c_-\Ga(-\nu_-)(\lp^{\nu_-}-(\lp+i\xi)^{\nu_-})
\eqa
as $\xi\to\infty$ in $i\cC_\pi$ and $-i\cC_\pi$, respectively:
\begin{eqnarray}\label{asKBLbasicp}
\psi^0_+(\nu_+, c_+, \lm ;\rho e^{i\varphi})&=&-c_+\Ga(-\nu_+)\exp[i(-\pi/2+\varphi)\nu_+]\rho^{\nu_+}
+O(\rho^{\nu_+-1}),\\
\label{asKBLbasicm}
\psi^0_-(\nu_-, c_-, \lp ;\rho e^{i\varphi})&=&-c_-\Ga(-\nu_-)\exp[i(\pi/2+\varphi)\nu_+]\rho^{\nu_-}
+O(\rho^{\nu_--1}).
\eqa
Let $\varphi\in (-\pi/2,\pi/2)$. We have 
\bbe\label{repsi0nuneq12p}
\Re (-\Ga(-\nu)\exp[i(-\pi/2+\varphi)\nu])(=-\Ga(-\nu)\cos((-\pi/2+\varphi)\nu))>0
\ee
iff  either
\begin{enumerate}[(i)]
\item  $\nu\in (1,2)$ (hence, $-\Ga(-\nu)<0$) and $-3\pi/2<(-\pi/2+\varphi)\nu<-\pi/2$; 
equivalently, $\varphi\in (\max\{-1,(1-3/\nu)\}\pi/2, (1-1/\nu)\pi/2))$, or
\item 
$\nu\in (0,1)$ (hence, $-\Ga(-\nu)>0$) and $-\pi/2<(-\pi/2+\varphi)\nu<\pi/2$,
equivalently, \\ $\varphi\in (\max\{-1, (1-1/\nu)\}\pi/2\},\pi/2](\subset ((1-1/\nu)\pi/2, (1+1/\nu)\pi/2))$.
Hence, \[
\cC_+=\{\rho e^{i\varphi}\ |\ \rho>0, \max\{-1, (1-1/\nu)\}\pi/2<\varphi< \pi-\max\{-1, (1-1/\nu)\}\pi/2\}=i\cC_\ga,
\]
where $\ga=\pi/2-\max\{-1, (1-1/\nu)\}\pi/2=\min\{1,1/(2\nu)\}\pi/2$.
\end{enumerate}
Similarly,
\bbe\label{repsi0nuneq12m}
\Re (-\Ga(-\nu)\exp[i(\pi/2+\varphi)\nu])(=-\Ga(-\nu)\cos((\pi/2+\varphi)\nu))>0
\ee
if and only if either 
\begin{enumerate}[(i)]
\item $\nu\in (1,2)$ (hence, $-\Ga(-\nu)<0$) and $\pi/2<(\pi/2+\varphi)\nu<3\pi/2$; 
equivalently, $\varphi\in ((1/\nu-1)\pi/2, \min\{1, 3/\nu_--1)\}\pi/2))$, or
\item 
$\nu\in (0,1)$ (hence, $-\Ga(-\nu)>0$) and $-\pi/2<(\pi/2+\varphi)\nu<\pi/2$,
equivalently, $\cC_+=-i\cC_{\min\{1,1/(2\nu_-)\}\pi/2}$.

\end{enumerate}
All statements (a)-(e) are immediate from the properties of $\psi^0_\pm$ established above.

\end{proof}

\subsubsection{KoBoL processes of order 1}\label{OnesidedKoBoL1}  Formulas for $\psi^0$ are derived in \cite{KoBoL,NG-MBS}. We have to consider several sets of conditions on the parameters on the RHS of \eq{KBLmeqdifnu}.
\begin{enumerate}[(1)]
\item
If
$c_-=0$ and $\nu_+=1$, $c_+>0$, $\lm\le  0$, we have
\bbe\label{KBL1p}
\psi^0_+(1;\la;\xi)=c_+((-\lm)\ln(-\lm)-(-\lm-i\xi)\ln(-\lm-i\xi)),
\ee
and, therefore, for $\varphi\in (-\pi/2,\pi/2)$, as $\rho\to+\infty$,
\bbe\label{KBL1pasymp}
\psi^0_+(1;\la; \rho e^{i\varphi})=c_+e^{i(\pi/2+\varphi)}\rho\ln\rho +O(\ln \rho).
\ee
 Hence, $X$ is  SINH-regular of order $1+$ and type $((\lm,+\infty), \bC\setminus i(-\infty,0],
 \cC_{-\pi/2,0})$.
 \item
If
$c_+=0$ and $\nu_-=1$, $c_->0$, $\lp\ge  0$, we have
\bbe\label{KBL1m}
\psi^0_-(1;\la;\xi)=c_-(\lp\ln\lp-(\lp+i\xi)\ln(\lp+i\xi)),
\ee
and, therefore, for $\varphi\in (-\pi/2,\pi/2)$, as $\rho\to+\infty$,
\bbe\label{KBL1pasymm}
\psi^0_-(1;\la; \rho e^{i\varphi})=c_-e^{i(-\pi/2+\varphi)}\rho\ln\rho +O(\ln \rho).
\ee
 Hence, $X$ is  SINH-regular of order $1+$ and type $((-\infty,\lp), \bC\setminus i[0,+\infty),
 \cC_{0,\pi/2})$.
 \item
 If $c_+, c_->0$, $\nu_-=\nu_+=1$, and either $\lm<0\le \lp$ or $\lm\le 0<\lp$, then 
 \beqa\label{KBLnueq11}
 \psi^0(\xi)&=& c_+[(-\lm)\ln(-\lm)-(-\lm- i\xi)\ln(-\lm- i\xi)]\\\nonumber &&+c_-[\lp\ln\lp-(\lp+ i\xi)\ln(\lp+
i\xi)].
\eqa
Hence, 
for $\varphi\in (-\pi/2,\pi/2)$, as $\rho\to+\infty$,
\bbe\label{KBLnueq11asymp}
\psi^0(\rho e^{i\varphi})=e^{i\varphi}\left((c_+-c_-)i\rho\ln\rho +(c_++c_-)\pi\rho\right) +O(\ln \rho).
\ee
Therefore, $X$ is  a SINH-regular process of\begin{enumerate}[(i)]
\item
   order $1+$ and type $((\lm,\lp), \bC\setminus i\bR,
 \cC_{-\pi/2,0})$, if $c_+>c_-$;
\item
  order $1+$ and type $((\lm,\lp), \bC\setminus i\bR,
 \cC_{0,\pi/2})$, if $c_->c_+$;
 \item
 order $1$ and type $((\lm,\lp), \bC\setminus i\bR,
 \bC\setminus i\bR)$, if $c_+=c_->0$.
 \end{enumerate}

 \end{enumerate}
 
\subsubsection{KoBoL processes of order $0+$}\label{OnesidedKoBoL0}  The formulas for $\psi^0$ are derived in \cite{KoBoL,NG-MBS}. We have to consider several sets of conditions on the parameters on the RHS of \eq{KBLmeqdifnu}.
\begin{enumerate}[(1)]
\item
If 
$c_-=0$ and $\nu_+=0$, $c_+=c>0$, $\lm\le  0$, we have
\bbe\label{KBL0p}
\psi^0_+(\xi)=c(\ln(-\lm-i\xi)-\ln(-\lm)),
\ee
and, therefore, for $\varphi\in (-\pi/2,3\pi/2)$, as $\rho\to+\infty$,
\bbe\label{KBL0pasymp}
\psi^0_+(\rho e^{i\varphi})=c\ln\rho +O(1).
\ee
 Hence, $X$ is  SINH-regular  of order $0+$ and type $((\lm,+\infty), \bC\setminus i(-\infty,0],
 \bC\setminus i(-\infty,0])$.
 \item
If 
$c_+=0$ and $\nu_-=0$, $c_-=c>0$, $\lp\ge  0$, we have
\bbe\label{KBL0m}
\psi^0_-(\xi)=c(\ln(\lp+i\xi)-\ln(\lp)),
\ee
and, therefore, for $\varphi\in (-3\pi/2,\pi/2)$, as $\rho\to+\infty$, \eq{KBL0pasymp} holds. Hence, 
$X$ is  SINH-regular  of order $0+$ and type $((-\infty,\lp), \bC\setminus i[0+\infty),
 \bC\setminus i[0+\infty))$.
 \item
Let  $\nu_\pm=0$, $c_\pm>0$, and either $\lm<0\le \lp$ or $\lm\le 0<\lp$. Then
\bbe\label{KBL_0+}
\psi^0(\xi)=c_+(\ln(-\lm-i\xi)-\ln(-\lm))+c_-(\ln(\lp+i\xi)-\ln\lp),
\ee
and, therefore, for $\varphi\in (-\pi/2,\pi/2)$, as $\rho\to+\infty$, \eq{KBL0pasymp} holds with $c=c_++c_-$. Hence, 
$X$ is  SINH-regular  of order $0+$ and type $((\lm,\lp), \bC\setminus i\bR,
 \bC\setminus i\bR)$. Note that if $c_+=c_-$, \eq{KBL_0+} defines the characteristic exponent of a VGP. 
\end{enumerate}

 \subsection{The order and type of processes of the $\be$-class}\label{Calculations in the beta-model} 
      To calculate the order of the process and find $\cC_+$, 
recall that if $z\to\infty$ in the sector $|\mathrm{arg}\,z|\le \pi-\de$, where $\de>0$, then, for any $a,b\in \bC$, 
\bbe\label{as_rat_Ga}
\frac{\Ga(z+a)}{\Ga(z+b)}\sim z^{a-b}.
\ee
Hence, for any $\de>0$, as $\xi\to \infty$ in $\cC_{-\pi/2+\de,\pi/2-\de}$,  
\bbe\label{as_rat_Ga2}
\frac{\Ga(\pm i\xi+a)}{\Ga(\pm i\xi+b)}\sim (\pm i\xi)^{a-b},
\ee
and we conclude that, as $\xi\to \infty$ in $\cC_{\pi/2-\de}$, 
\beqa\label{asBetaChExp}
\psi^0(\xi)&=&\frac{\sg^2}{2}\xi^2 -\frac{c_1}{\be_1}\Ga(1-\ga_1)(-i\xi)^{-1+\ga_1}(1+o(1))\\
\nonumber && \hskip1.1cm-\frac{c_2}{\be_2}\Ga(1-\ga_2)(i\xi)^{-1+\ga_2}(1+o(1))
+O(1).
\eqa
Consider the following cases.
\begin{enumerate}[(i)]
\item
$\sg^2>0$. Since $\ga_1,\ga_2<3$, the first term on the RHS of \eq{asBetaChExp} is the leading term
of asymptotics, hence, $X$ is  SINH-regular of order 2.
\item
$\sg^2=0$,  $\ga_1>\ga_2$, $c_1>0$. The leading term of asymptotics is the second term on the RHS of
\eq{asBetaChExp}. Fix $\varphi:=\mathrm{arg}\, \xi\in (-\pi/2,\pi/2)$, and let $\rho=|\xi|\to +\infty$.  We have
\bbe\label{asBeta12}
\psi^0(\xi)\sim c^1_\infty(\varphi)\rho^{-1+\ga_1},
\ee
where $c^1_\infty(\varphi)=-(c_1/\be_1)\Ga(1-\ga_1)e^{i(-\pi/2+\varphi)(-1+\ga_1)}$.
Since 
\[
\Re c^1_\infty(\varphi)=-(c_1/\be_1)\Ga(1-\ga_1)\cos((-\pi/2+\varphi)(-1+\ga_1)),\]
\begin{itemize}
\item
in the case $\ga_1\in (1,2)$,   we have $-\Ga(1-\ga_1)>0$, hence, $\Re c_\infty(\varphi)>0$ iff $-\pi/2<(-\pi/2+\varphi)(\ga_1-1)<\pi/2$.
Thus, $X$ is  of order $\ga_1-1$, and $\cC_+=\cC_{\gam,\gap}$, where $\gam=\max\{-1, -1/(\ga_1-1)+1\}\pi/2$ and $\gap=\pi/2$;
\item
in the case $\ga_1\in (2,3)$,   we have $-\Ga(1-\ga_1)<0$, hence,
 $\Re c_\infty(\varphi)>0$ iff $-3\pi/2<(-\pi/2+\varphi)(\ga_1-1)<-\pi/2$.
Thus, $X$ is  of order $\ga_1-1$, and $\cC_+=\cC_{\gam,\gap}$, where
 $\gam=\max\{(-3/(\ga_1-1)+1),-1\}\pi/2$ and $\gap=(1-1/(\ga_1-1))\pi/2$.
\end{itemize}
\item
$\sg^2=0$,  $\ga_2>\ga_1$, $c_2>0$. The leading term of asymptotics is the third term on the RHS of
\eq{asBetaChExp}. Fix $\varphi:=\mathrm{arg}\, \xi\in (-\pi/2,\pi/2)$, and let $\rho=|\xi|\to +\infty$.  We have
\bbe\label{asBeta21}
\psi^0(\xi)\sim c^2_\infty(\varphi)\rho^{-1+\ga_2},
\ee
where $c^2_\infty(\varphi)=-(c_2/\be_2)\Ga(1-\ga_2)e^{i(\pi/2+\varphi)(-1+\ga_2)}$.
Since 
\[
\Re c^2_\infty(\varphi)=-(c_2/\be_2)\Ga(1-\ga_2)\cos((\pi/2+\varphi)(-1+\ga_2)),\]
\begin{itemize}
\item
in the case $\ga_2\in (1,2)$,   we have $-\Ga(1-\ga_2)>0$, hence,
 we have $\Re c_\infty(\varphi)>0$ iff $-\pi/2<(\pi/2+\varphi)(\ga_2-1)<\pi/2$.
 Thus, $X$ is  of order $\ga_2-1$, and $\cC_+=\cC_{\gam,\gap}$, where $\gam=-\pi/2$ and $\gap=\min\{1, 1/(\ga_1-1)-1\}\pi/2$;
\item
in the case $\ga_2\in (2,3)$,   we have $-\Ga(1-\ga_2)<0$, hence,
 $\Re c_\infty(\varphi)>0$ iff $\pi/2<(\pi/2+\varphi)(\ga_2-1)<3\pi/2$.
Thus, $X$ is  of order $\ga_2-1$, and $\cC_+=\cC_{\gam,\gap}$, where
 $\gam=(1/(\ga_2-1)-1)\pi/2$ and $\gap=\min\{1,(3/(\ga_2-1)-1)\}\pi/2$.
 \end{itemize}
 \item
 $\sg^2=0$,  $\ga_1=\ga_2=\ga\in (1,3), \ga\neq 2$, $c_1,c_2>0$. The leading term of asymptotics is the sum of the second and 
 third term on the RHS of
\eq{asBetaChExp}. Fix $\varphi:=\mathrm{arg}\, \xi\in (-\pi/2,\pi/2)$, and let $\rho=|\xi|\to +\infty$.  We have
\bbe\label{asBeta2eq1}
\psi^0(\xi)\sim c_{\ga;\infty}e^{i\varphi(\ga-1)}\rho^{-1+\ga},
\ee
where 
$
c_{\ga;\infty}=-\Ga(1-\ga)((c_1/\be_1)e^{-i(\ga-1)\pi/2}+(c_2/\be_2)e^{i(\ga-1)\pi/2}).
$

 
 If $\ga\in (1,2)$, then $-\Ga(1-\ga)>0$ and $\Re e^{\pm i(\ga-1)\pi/2}>0$.
 If $\ga\in (2,3)$, then $-\Ga(1-\ga)<0$ and $\Re e^{\pm i(\ga-1)\pi/2}<0$. In both cases, $\Re c_{\ga;\infty}>0$, 
$\varphi_\ga:=\mathrm{arg}\, c_{\ga;\infty}\in (-\pi/2,\pi/2)$ and $\Re (c_{\ga;\infty}e^{i(\varphi\varphi_0)(\ga-1)})>0$ iff
$-\pi/2<(\varphi_0+\varphi)(\ga-1)<\pi/2$.
Therefore, $X$ is  of order $\ga-1$, and $\cC_+=\cC_{\gam,\gap}$, where
 $\gam=\max\{-\pi/2, -\pi/(2(\ga-1))-\varphi_0\}$ and $\gap=\min\{\pi/(2(\ga-1))-\varphi_0,\pi/2\}$.
\item
If $\sg^2=0$ and $\ga_1,\ga_2\in (0,1)$, then $X$ is a compound Poisson process, which we do not consider in this paper.
\end{enumerate}

\subsection{Calculation of the residues in the $\be$-model}\label{ss:calc_res_beta}
From \eq{BetaChExp}, we see that the poles in the upper half-plane are at $\xi_n, n=0,1,\ldots,$ defined by $\al_2+\frac{i\xi_n}{\be_2}=-n$.
      Hence, $\xi_n=i(\al_2+n\be_2)$. Using the relation between the Beta and Gamma functions and reflection formula for the latter,
      we obtain
      \beqast
      B(\al_2+\frac{i\xi}{\be_2},1-\ga_2)&=&\frac{\Ga(\al_2+\frac{i\xi}{\be_2})\Ga(1-\ga_2)}{\Ga(\al_2+\frac{i\xi}{\be_2}+1-\ga_2)}\\
      &=&\frac{\pi\Ga(1-\ga_2)\Ga(-\al_2-\frac{i\xi}{\be_2}+\ga_2)\sin(\pi(\al_2+\frac{i\xi}{\be_2}+1-\ga_2))}{\sin(\pi(\al_2+\frac{i\xi}{\be_2}))
      \Ga(1-\al_2-\frac{i\xi}{\be_2})\pi}.
      \eqast
Letting $\xi=\xi_n+z$, where $z\to 0$, we obtain $\al_2+\frac{i\xi_n}{\be_2}=-n+\frac{iz}{\be_2}$, and, therefore,
   \beqast
      B(\al_2+\frac{i\xi}{\be_2},1-\ga_2)&=&\frac{\Ga(1-\ga_2)\Ga(n-\frac{iz}{\be_2}+\ga_2)\sin(\pi(-n+\frac{iz}{\be_2}+1-\ga_2))}{\sin(\pi(-n+\frac{iz}{\be_2}))
      \Ga(1+n-\frac{iz}{\be_2})}\\
      &=&\frac{\Ga(1-\ga_2)\Ga(n+\ga_2)\sin(\pi(1-\ga_2))}{\frac{iz}{\be_2}
      \Ga(1+n)}+O(1),\ z\to 0.
      \eqast  
      Calculating the  residues, we find
      \bbe\label{eq:posdensity_beta}
      f_+(x)= \sum_{n=0}^{+\infty} e^{-a_n x}a_n p_n,
      \ee
      where $a_n=\al_2+n\be_2$ and 
      \[
    a_n p_n=\frac{\be_2\Ga(1-\ga_2)\Ga(n+\ga_2)\sin(\pi(1-\ga_2))}{\Ga(n+1)}.
      \]
      As $n\to+\infty$, $\Ga(n+\ga_2)/\Ga(n+1)\sim n^{\ga_2-1}$. Since
       $\ga_2\in (0,3), \ga_2\not\in\{1,2\}$, we have      
     $p_n\sim C a_n^{\ga_2-2}$ as  $n\to+\infty,$
     where $C$ is independent of $n$. Thus, the $\be$-model is a meromorphic process with $\al_n=a_n$ and
     $p_n$ satisfying 
     \bbe\label{eq:beta_series_bound}
     \sum_{n=0}^{+\infty} p_n (\al_n)^{-2}\le C \sum_{n=0}^{+\infty} (\al_n)^{\ga_2-4}<+\infty.
     \ee
By symmetry, we calculate the densities of negative jumps are calculated and study the convergence of the series.

\subsection{Proof of Lemma \ref{lem:der_SINH}}\label{ss:proof_lem:der_SINH}
For  a given domain of the form $i[\mumpr,\muppr]+(\cC'\cup\{0\})$, there exist $c>0$, $\mum^{\prime\prime}
 \in (\mum,\mumpr), \mup^{\prime\prime}
 \in (\muppr,\mup)$ and closed cone $\cC^{\prime\prime}\subset \cC\cup\{0\}$ such that, for each 
 $\xi\in i[\mumpr,\muppr]+(\cC'\cup\{0\})$,
 the ball $B(\xi, c(1+|\xi|))$  of radius $c(1+|\xi|)$, centered at $\xi$, is a subset of
  $i[\mum^{\prime\prime},\mup^{\prime\prime}]+(\cC^{\prime\prime}\cup\{0\})$. It follows from the definition of SINH-regular processes
that there exists $C_0>0$ such that, for all $\xi\in i[\mum^{\prime\prime},\mup^{\prime\prime}]+\cC^{\prime\prime}$ and $j=0$,
the bound \eq{der_psi_nu02} holds. It follows that there exists $C>0$ s.t.,
 if $\nu\in (0,2]$, then, for $\xi\in i[\mumpr,\muppr]+(\cC'\cup\{0\})$ and $j=0$, the following modified version of \eq{der_psi_nu02}
 holds:
 \bbe\label{der_psi_nu02mod}
 \max_{\eta\in B(\xi, c(1+|\xi|))}|\psi(\eta)|\le C (1+|\xi|)^{\nu}.
 \ee
  By the Cauchy integral theorem, for any
 $\xi\in [\mumpr,\muppr]+(\cC'\cup\{0\})$,
 \bbe\label{repr_in_cone}
 \psi(\xi)=\frac{1}{2\pi i}\int_{\dd B(\xi, c(1+|\xi|))}\frac{\psi(\eta)}{\xi-\eta}d\eta.
 \ee
 Differentiating under the integral sign $j$ times and using \eq{der_psi_nu02mod},  we obtain
\[
 |\psi^{(j)}(\xi)|\le C j! c^{-j-1}(2+|\xi|)^{\nu-j-1}\frac{1}{2\pi}\int_{\dd B(\xi, c(1+|\xi|))}|d\eta|
 \le C_1 j!c^{-j}(1+|\xi|)^{\nu-j},
 \]
 which proves the bound \eq{der_psi_nu02}. 

\subsection{Proof of Lemma \ref{lim_0}}\label{ss:proof_of_lim_0}
 We have 

\[
\int_{|z+\mu|=\eps, \Re z>-\mu}ST(\cG)(z)dz=\int_{\mu}^{+\infty} \ln\frac{t-\mu+i\eps}{t-\mu-i\eps}\cG(dt).
\]
Choose a function  $\de=\de(\eps)$ such that $\de(\eps)\to 0+$ and $\eps/\de(\eps)\to 0$ as  $\eps\to 0+$.
The integrand on the RHS above is uniformly bounded and $\cG((\mu,\mu+\de))\to 0$ as $\de\to 0$, hence, the integral over $(\mu,\mu+\de)$ tends to 0. On $(\mu+\de, 1)$, the integrand is $O(\eps/\de)$, and $\cG([\mu+\de,1])<\infty$, hence, the integral over $(\mu+\de, 1)$ 
tends to 0 as well. Finally, on $[1,+\infty)$, the integrand is bounded by $C\eps/t$, where $C$ is independent of $t, \eps$, and $\int_{[1,+\infty)} t^{-1}\cG(dt)<\infty$, hence, the integral over $[1,+\infty)$ tends to 0.

\subsection{Proof of Theorem \ref{thm1:SL} (a)}\label{ss:proof_thm1:SL}
 Let $\Im\xi\ge 0$ and $a_2=1, a_1=0$. Then we calculate
\beqast
\psi_+(\xi)&=&\int_0^{+\infty}(1-e^{ix\xi}+\bfo_{(0,1]}(x)ix\xi)\int_{(0,+\infty)} e^{-tx}t^2 \cG(dt) dx\\
&=&\int_{(0,+\infty)} \cG(dt)\left[t^2\int_0^{+\infty}e^{-tx}dx-t^2\int_0^{+\infty}e^{-(t-i\xi)x}dx+i\xi t^2 \int_0^1 x e^{-tx}dx\right]\\
&=&\int_{(0,+\infty)} \cG(dt)\left[t-\frac{t^2}{t-i\xi}+i\xi (1-e^{-t}(1+t))\right]\\
&=&\xi^2\int_{(0,+\infty)} \frac{\cG(dt)}{t-i\xi}-i\xi \int_{(0,+\infty)} \cG(dt)e^{-t}(1+t).\eqast
(To prove that Fubini's theorem is applicable, one represent the integral w.r.t. $x$ as the sum of integrals over $(0,1]$ and $(1,+\infty)$, and uses the bound $|1-e^{ix\xi}+\bfo_{(0,1]}(x)ix\xi|\le C(\eps)x^2, 0<x<1,$ to prove the absolute convergence).
Similarly, if $a_2=0, a_1=1$, we calculate
\beqast
\psi_+(\xi)&=&\int_0^{+\infty}(1-e^{ix\xi})\int_{(0,+\infty)} e^{-tx}t \cG(dt) dx\\
&=&\int_{(0,+\infty)} \cG(dt)\left[t\int_0^{+\infty}e^{-tx}dx-t\int_0^{+\infty}e^{-(t-i\xi)x}dx\right]\\
&=&\int_{(0,+\infty)} \cG(dt)\left[1-\frac{t}{t-i\xi}\right]
=-i\xi\int_{(0,+\infty)} \frac{\cG(dt)}{t-i\xi}.\eqast

\subsection{Calculation of the (s)SL-measures for Generalized Hyperbolic distributions}\label{hyper_SL_calc}
    In \eq{eq:modBessel3}, we change the variables $e^t=y+\sqrt{y^2-1}, dt=\frac{dt}{\sqrt{y^2-1}}$, then $y=y'+1$:
    \bbe\label{eq:modBessel3b}
    K_\la(z)=e^{-z}\int_0^\infty e^{-zy'}\frac{(y'+1+\sqrt{(y')^2+2y'})^\la
    +(y'+1+\sqrt{(y')^2+2y'})^{-\la}}{2\sqrt{(y')^2+2y'}}dy'.
    \ee
    We apply \eq{eq:modBessel3b} for $z=\de \sqrt{\al^2-(\be+i\xi)^2}$ for $\xi=it+\eps$, where 
    $t>\al+\be$ and $\eps\downarrow 0$ or $t<-\al+\be$ and $\eps\uparrow 0$. By symmetry, it suffices to consider the case $t>\al+\be$. We represent $z$ in the form 
    $z=\rho e^{i\varphi}$, where $\varphi \in (0,\pi/2)$ and $\rho=\rho_+(t)=\de\sqrt{(\al-\be+t)(t-\al-\be)}$, fix $\varphi$, rotate the line of integration to $e^{-i\varphi}\bR_+$,
    and change the variable $y'=e^{-i\varphi}u$. The result is
    \bbe\label{Kla_int}
    K_\la(\rho e^{i\varphi})=e^{-\rho e^{i\varphi}}I_\la(\rho,\varphi),
    \ee
 where
\[ I_\la(\rho,\varphi)=\int_0^{+\infty}e^{-\rho u}\frac{e^{-i\varphi\la}(u+e^{i\varphi}+\sqrt{u^2+2ue^{i\varphi}})^\la+
e^{i\varphi\la} (u+e^{i\varphi}+\sqrt{u^2+2ue^{i\varphi}})^{-\la}}{2\sqrt{u^2+2ue^{i\varphi}}}du.
 \]
 If $\la=1$ (the case of Hyperbolic processes), the integral simplifies:
  \bbe\label{Kla_-int2_la=1}
 I_1(\rho,\varphi)=e^{-i\varphi}\int_0^{+\infty}e^{-\rho u}\frac{u+e^{i\varphi}}{\sqrt{u^2+2ue^{i\varphi}}}du.
 \ee
 In the limit $\varphi=\pi/2-0$, we have
\[
 I_\la(\rho,\pi/2-0)=\int_0^{+\infty}e^{-\rho u}\frac{e^{-i\la\frac{\pi}{2}}(u+i+\sqrt{u^2+2ui})^\la+
e^{i\la\frac{\pi}{2}} (u+i+\sqrt{u^2+2u i})^{-\la}}{2\sqrt{u^2+2ui}}du,
   \]
   and 
   \bbe\label{Kla_-int2_la=1_lim}
 I_1(\rho,\pi/2-0)=-i\int_0^{+\infty}e^{-\rho u}\frac{u+i}{\sqrt{u^2+2ui}}du.
 \ee
 The density of the measure $\cG_+(\la;dt)=\bfo_{[\al+\be,+\infty)}(t)g_+(\la;t)dt$ is
\beqa\label{glimp}
g_+(\la;t)&=&-\Im \left[\ln I_\la(\rho_+(t),\pi/2-0)\right]+(1-\la)\rho_+(t)\pi/2\\\nonumber
&=&(1-\la)\frac{\pi}{2}\rho_+(t)-\mathrm{arg}\, I_\la(\rho(t),\pi/2-0).
\eqa
 It is easy to see that in cases $\la=0$ or $\la=1$, $\Im I_\la(\rho,\pi/2-0)<0$, hence,
 \[
\Im( -\ln I_\la(\rho,\pi/2-0))=-\mathrm{arg}\, I_\la(\rho,\pi/2-0)>0,
\]
and $g_+(\la;t)>0, t>\al+\be$. By symmetry, $\cG_-(\la;dt)=\bfo_{[\al-\be,+\infty)}(t)g_(\la;t)dt$, where
\bbe\label{glimm}
g_-(\la;t)=(1-\la)\frac{\pi}{2}\rho_-(t)-\mathrm{arg}\, I_\la(\rho_-(t),\pi/2-0),
\ee
where $\rho_-(t)=\de\sqrt{(t-\al+\be)(\al+\be+t)}$. Hence, if $\la=0,1$, the distribution $\sg(dx)$ is a SL distribution. 

For  $\la\in [-2,2]$,
we verify the condition $\Im I_\la(\rho,\pi/2-0)<0$ numerically checking that
for all $y>0$, the imaginary part of the integrand is negative, hence, if $\la\in [-2,1]$, $\sg(dx)$ is a regular SL-distribution. If $\la>1$, then, for large $t$, $g_\pm(\la,t)<0$, and $\sg(dx)$ is not a  SL-distribution.
If $\la<-2$, the imaginary part of integrand changes sign, hence, we cannot obtain a definitive result by these simple means. We only note that if $\la<-2$, then, as $u\to+\infty$, the leading term of asymptotics of the imaginary part of the fraction in front of $e^{-\rho u}$ is $-\sin(\la\pi/2)u^{-\la-1}$, hence, if $-\la\pi/2\in (1,2)\cup(3,4)\cup\cdots$, then,
as $\rho_\pm(t)\downarrow 0$, $g_\pm(\la,\rho_\pm)\to -\infty$. It follows that $\sg(dx)$ is not a SL-distribution if $\la<-2$ and either $\al-\be$ or $\al+\be$ are sufficiently close to 0 or $\de$ is. The question whether $\la<-2$ implies that $\sg(dx)$ is not
a SL distribution for all admissible $\de,\al,\be$ remains open.

    \subsection{Proof of  Theorem \ref{thm:one-sided_SL_SINH}}\label{ss:proof_thm:one-sided_SL_SINH}
 Fix $\ga\in (0,\pi/2)$, and let $\xi=\rho e^{i\varphi}$, where $\rho>0$ and $|\varphi|\le \ga$.
We have  
\beqast
\Re\psi_-(\xi)&=& \Re\int_\mu^{+\infty}\frac{\xi^2 \cG(dt)}{t+i\xi}\\
&=&\rho^2\Re\left\{(\cos(2\varphi)+i\sin(2\varphi))\int_\mu^{+\infty}\frac{(t-\rho\sin\varphi-i\rho\cos\varphi)\cG(dt)}{(t-\rho\sin\varphi)^2+\rho^2\cos^2\varphi}\right\}\\
&=&\rho^2 (\cos(2\varphi)I_1(\cG;\mu,\varphi;\rho)+\rho\sin\varphi  I_0(\cG;\mu,\varphi;\rho)),
\eqast
where
\beqast
I_k(\cG;\mu,\varphi;\rho)&:=&\int_\mu^{+\infty}\frac{t^k\cG(dt)}{(t-\rho\sin\varphi)^2+\rho^2\cos^2\varphi}.
\eqast
On the strength of \eq{eq:bound_SL_SINH_at},
 there exist $C_1,c_1, \rho_0>0>0$  such that for  $k=0,1,$ $\varphi\in [-\ga,\ga]$, and $\rho\ge \rho_0$,
\bbe\label{ineq:double_bound}
c_1\rho^{\nu+k-3}\int_{\mu/\rho}^{+\infty}\frac{t^{k+\nu-2}dt}{t^2+1}\le
I_k(\cG;\mu,\varphi;\rho)\le C_1 \rho^{\nu+k-3}\int_{\mu/\rho}^{+\infty}\frac{t^{k+\nu-2}dt}{t^2+1}.
\ee
(a) If $\nu=\al+2\in (1,2)$, 
\bbe\label{cknu12}
\int_{\mu/\rho}^{+\infty}\frac{t^{k+\nu-2}dt}{t^2+1}=c(k;\nu)+o(1),\ \rho\to +\infty,
\ee
where
\[
c(k;\nu)=\int_{0}^{+\infty}\frac{t^{k+\nu-2}dt}{t^2+1},
\]
therefore there exists $C_1$ independent of $\rho$ and $\varphi$ s.t.
\[
\Re\psi_-(\xi)\ge  \rho^\nu[c_1(c(1;\nu)\cos(2\varphi)+c(0,\nu)(\sin\varphi)_+)+C_1c(0,\nu)(\sin\varphi)_- ]+o(\rho^\nu)).
\]
Evidently, there exist $-\pi/2<\gam<0<\gap<\pi/2$ s.t. for all $\varphi\in (\gam,\gap)$,
\[
c_1c(1,\nu)(\cos(2\varphi)+(\sin\varphi)_+)+C_1c(0,\nu)(\sin\varphi)_- >0,\]
hence,  \eq{lowerbound_gen}  holds with $\cC_+=\cC_{\gampr,\gappr}$, for any $[\gampr,\gappr]\subset (\gam,\gap)$.
The upper bound $|\psi_-(\xi)|\ge C(1+|\xi|)^\nu$ is proved similarly. 

(b) If $\nu=1$, then, for $k=1$, \eq{cknu12} holds but 
\[
\int_{\mu/\rho}^{+\infty}\frac{t^{\nu-2}dt}{t^2+1}\sim \ln\rho+O(1),\ \rho\to +\infty,
\]
hence,
\[
\Re\psi_-(\xi)\ge  \rho[c_1(\c(1;1)\cos(2\varphi)+(\sin\varphi)_+\ln\rho)+(\ln\rho) C_1(\sin\varphi)_- ]+o(\rho^\nu)).
\]
It follows that there exist $c>0$ and  $\gap>0$ s.t. for $0\le \varphi\le \gap$,
\[
\Re\psi_-(\xi)\ge c\rho(1+(\sin\varphi)_+\ln\rho)
\]
(but if $\varphi<0$, then $\Re\psi_-(\xi)\sim -c(\varphi)\rho\ln\rho$, where $c(\varphi)>0$). 
The upper bound \[
|\psi_-(\xi)|\le C(1+|\xi|)^\nu \ln(2+|\xi|)\]
 is proved similarly, and we conclude that $X$ is SINH-regular of order $(1,1+)$.

(c) Let  $\nu=\al+1\in (0,1)$. Then we consider
\beqast
\Re\psi_-(\xi)&=& \Re\int_\mu^{+\infty}\frac{i\xi\cG(dt)}{t+i\xi}\\
&=&\rho\Re\left\{(i\cos(\varphi)-\sin(\varphi))\int_\mu^{+\infty}\frac{(t-\rho\sin\varphi-i\rho\cos\varphi)\cG(dt)}{(t-\rho\sin\varphi)^2+\rho^2\cos^2\varphi}\right\}\\
&=&\rho (-\sin\varphi I_1(\cG;\mu,\varphi;\rho)+\rho\cos(2\varphi ) I_0(\cG;\mu,\varphi;\rho)).
\eqast
The bound \eq{ineq:double_bound} assumes the form
\[
c_1\rho^{\nu+k-2}\int_{\mu/\rho}^{+\infty}\frac{t^{k+\nu-1}dt}{t^2+1}\le
I_k(\cG;\mu,\varphi;\rho)\le C_1 \rho^{\nu+k-2}\int_{\mu/\rho}^{+\infty}\frac{t^{k+\nu-1}dt}{t^2+1},
\]
and we continue trivially modifying the proof in the case $\nu\in (1,2)$.

(d) is proved similarly.

\subsection{Sufficient conditions for the existence of zeros on $i(\mum\mup)$}\label{ss:proof_of_Lemma_lem:lim0psi}
\begin{lem}\label{lem:lim0psi} Let $\cG\in \SM_\mu$, $\mu\ge 0.$   
Then, as $\eps\to 0+$, uniformly in $\varphi\in [-\pi/2,\pi/2]$,
 \begin{enumerate}[(i)]
 \item
$\Re ST(\cG)(-\mu+\eps e^{i\varphi})\to +\infty$ if $A:= \int_{[\mu,+\infty)}(t-\mu)^{-1}\cG(dt)=+\infty$; 
  \item
 if $A_1:=\int_{(\mu,+\infty)}(t-\mu)^{-1}\cG(dt)<+\infty$, then, as $\eps\downarrow 0$,\[ 
 \Re ST(\cG)(-\mu+\eps e^{i\varphi})=\eps^{-1}\cG(\{\mu\})\cos\varphi+A_1+o(1);\]
 \item
 if $\cG_+$ (resp., $\cG_-$) satisfies the condition in (i), then, for any $q>0$, the equation $q+\psi(\xi)=0$ has
 a solution on $i(\mum,0)$ (resp., on $i(0,\mup))$.
 \end{enumerate}
  \end{lem}
  \begin{proof}   An atom at $\mu$ contributes 
 $\Re(\cG(\{\mu\})/(\mu-\mu+\eps e^{i\varphi}))=\eps^{-1}\cG(\{\mu\})\cos\varphi$, hence, it remains to consider $\cG(dt)$ without an atom at $\mu$.  For any $\eps>0$, function $\varphi\mapsto ST(\cG)(-\mu+\eps e^{i\varphi})$ is continuous on 
the compact $[-\pi/2,\pi/2]$, hence, it suffices to prove (i) and (ii) for a fixed $\varphi$. For $t>\mu$, 
\beqast
\Re \frac{1}{-\mu+\eps e^{i\varphi}+t}
&=&\Re \frac{t-\mu+\eps\cos\varphi-i\eps\sin\varphi}{(t-\mu+\eps\cos\varphi)^2+ \eps^2\sin^2\varphi}=
\frac{t-\mu+\eps\cos\varphi}{(t-\mu+\eps\cos\varphi)^2+ \eps^2\sin^2\varphi}.
\eqast 
Hence, 
\[\Re ST(\cG)(-\mu+\eps e^{i\varphi})< \int_{(\mu,+\infty)}(t-\mu+\eps\cos\varphi)^{-1}\cG(dt)<\int_{(0,+\infty)}(t-\mu)^{-1}\cG(dt),
\]
and $\Re (1/(-\mu+\eps e^{i\varphi}+t))\to 1/(-\mu+t)$ as $\eps\to 0+$, for any $t>\mu$. 
Hence, if $A_1<\infty$, then (ii) is valid on the strength of the dominated convergence theorem. 

Let  $A_1=\infty$. If $\cos\varphi=0$, then, for every $t$, 
$
\eps\mapsto  \Re \frac{1}{-\mu+\eps e^{i\varphi}+t}=\frac{t-\mu}{(t-\mu)^2+\eps^2}$ increases as $\eps\downarrow 0$.
Hence, if the integrals 
\[
I(\varphi,\eps)=\int_{(\mu,+\infty)}\Re \frac{1}{t-\mu+\eps e^{i\varphi}}\cG(dt)
\] are bounded by a finite constant $B(\varphi)$ independent of $\eps>0$, then $A\le B(\varphi)<\infty$ by Fatou's lemma, contradiction. If $\cos\varphi>0$, we use $ \eps^2(\sin\varphi)^2<(\tan\varphi)^2 (t-\mu+\eps\cos\varphi)^2$ to obtain the bound
\[
I(\varphi,\eps)(1+\tan^2\varphi)\ge I_1(\varphi,\eps):=\int_{(\mu,+\infty)}\frac{1}{t-\mu+\eps \cos\varphi}\cG(dt).
\]
The same argument as in the case $\cos\varphi=0$ shows that $I_1(\varphi,\eps)\to+\infty$, hence, $I(\varphi,\eps)\to+\infty$. Part (iii) is evident.

\end{proof}

 \subsection{Proof of of Lemma \ref{lem:limcuts}}\label{ss:proof of Lemma lem:limcuts}
   We prove the first bound; the second one follows from $\overline{ST(\cG)(z)}=ST(\cG)(\bar z)$.
   Let $\eps>0$. 
   We have $\Im (-b'+i\eps+t)^{-1}=-\eps ((t-b')^2+\eps^2)^{-1}$, therefore
   \beqast
  \Im \int_{(0,+\infty)}(-b'+i\eps+t)^{-1}\cG(dt)&=&-\eps \int_{(b-\de,b+\de)}((t-b')^2+\eps^2)^{-1}\cG(dt)+O(\eps)\\
  &\le &-c\eps \int_{(b-\de,b+\de)}((t-b')^2+\eps^2)^{-1}dt+O(\eps)\\
  &=& -c\int_{(b-\de-b')/\eps}^{(b+\de-b')/\eps}\frac{dt}{t^2+1}+O(\eps)\\
  &=&-c\int_{-\infty}^{+\infty}\frac{dt}{t^2+1}+O(\eps),
  \eqast
  where the constant for the $O$-terms can be chosen the same for all $b'\in (b-\de',b+\de')$.

\section{Outline of the conformal acceleration method}\label{s:CONF}

\subsection{Evaluation of probability distributions and pricing  European options}\label{ss:Europrice}
Efficient evaluation of probability distributions and expectations in L\'evy models are possible under certain regularity conditions
on $\psi$. The first two conditions are used in \cite{genBS,KoBoL,NG-MBS,barrier-RLPE,BLSIAM02} to define the class of Regular L\'evy Processes of Exponential type (RLPE):
(1) $F(dx)$ is absolutely continuous, and the L\'evy density $f(x)$ exponentially decays as $x\to \pm\infty$, equivalently, $\psi$ admits analytic continuation to a
strip $S_{(\mum,\mup)}$, where $\mum<0<\mup$; (2) $f(x)\sim c_\pm |x|^{-\nu_\pm-1}, x\to 0\pm$, where $c_\pm>0$
and $\nu_\pm\in [0,2)$. If $\nu_\pm\not\in\{0,1\}$, an equivalent condition is: as $\xi\to \infty$ in $S_{(\mum,\mup)}$, $\psi(\xi)$ stabilizes to
a positively homogeneous function (if $\nu_\pm\in\{0,1\}$, additional log-factor emerges). As it is remarked in op.cit., all classes of L\'evy processes used in finance enjoy properties (1) and (2); several classes of L\'evy processes constructed later also enjoy these properties. 
The importance of  properties (1)-(2)  are already seen in the simplest case 
of calculation of expectations of the form
\bbe\label{Europrice:def}
V(t,x)=\bE^{\bQ}\left[ e^{-r(T-t)}G(X_T)\ |\ X_t=x\right].
\ee
 (This is the price of the European option with the payoff $G(X_T)$ at the maturity date $T$, in the market with the constant riskless rate $r$, under the measure $\bQ$ chosen for pricing.)  Assuming that the Fourier transform 
\bbe\label{eq:defFT}
\hG(\xi)=(2\pi)^{-1}\int_\bR e^{-ix\xi}G(x)dx
\ee
is well-defined in a strip $S_{(\mumpr,\muppr)}$, where $\mum\le \mumpr<\muppr\le \mup$, we expand $G$ in the Fourier integral, substitute the result into \eq{Europrice:def}, apply Fubini's theorem to change the order of taking expectation and integration,
and obtain, for any $\om\in (\mumpr,\muppr)$, 
\bbe\label{Europrice:FT}
V(x,t)=(2\pi)^{-1}\int_{\Im\xi=\om} e^{i x'\xi+(T-t)(r+\psi^0(\xi))}\hG(\xi)d\xi,
\ee
where $x'=x+\mu(T-t)$ and $\mu$ is the one in  \eq{reprpsi}. 
The simplest way to evaluate the integral on the RHS of \eq{Europrice:FT} is to apply the simplified trapezoid rule.
The error of 
the infinite trapezoid rule decays as $\exp[-2\pi d/\ze]$, where
$d$ is the half-width of the strip of analyticity around the line of integration, and $\ze$ is the step. See, e.g., Thm. 3.2.1 in \cite{stenger-book}.  If the strip of analyticity is not too narrow, it is  easy to satisfy a very small error tolerance for the discretization error.
However, one must choose the line of integration and $\ze$ carefully otherwise serious errors may result. 
 Popular variations of this straightforward approach such as Carr-Madan method \cite{carr-madan-FFT} and COS method 
\cite{COS, COS2,COS3} introduce additional errors which lead to systematic errors in practically important situations. 
Lewis-Lipton formula is the standard Fourier inversion formula with the prefixed 
line of integration; the choice of the line is non-optimal in almost all cases and increases the CPU time. See \cite{iFT0,iFT,one-sidedCDS,MarcoDiscBarr,pitfalls,HestonCalibMarcoMe,HestonCalibMarcoMeRisk,SINHregular,BSINH} for details and numerical examples.

 \subsection{Conformal acceleration method}\label{ss:conf_accel}
In many cases of interest, the integrand decays slowly at infinity, and a very large number of terms of the truncated sum (simplified trapezoid rule) is needed to satisfy even a moderate error tolerance. However,
 in the case of standard European options, and in the case of piece-wise polynomial approximations of complicated payoffs
 \cite{single,MarcoDiscBarr,levendorskii-xie-Asian}, $\hG$ is meromorphic with a finite number of simple poles;
 in \cite{BSINH}, approximations with infinite number of poles appear.  If $X$ is SINH-regular of order $(\nu',\nu)$ with $\nu'>0$, one can use an appropriate conformal deformation and the corresponding change of variables to reduce calculations to the case of an integrand which is analytic in a strip around the line of integration and decays at infinity faster than exponentially. 
 Generally, the most efficient
change of variables ({\em sinh-acceleration}) is of the form $\xi=\chi_{\om_1, \om; b}(y)$, where $\om_1, \om \in \bR, b>0$ and
\begin{equation}\label{sinhbasic}
\chi_{\om_1, \om; b}(y)=i\om_1+b\sinh (i\om+y).
\end{equation}
If $x'>0$,   the contour of integration is deformed
 upwards, hence, $\om>0$ is used, if $x'<0$, the contour of integration is deformed
 downwards, hence, $\om<0$ is used. If $x'=0$, $\om$ of either sign can be used but, depending on
 the cone $\cC_+=\cC_{\gam,\gap}$, one of the choices is better. Typically, the choice
 $\om=(\gam+\gap)/2$ is optimal.  In \cite{SINHregular}, the reader can find  detailed general recommendation for the choice of the parameters
$\om_1,b,\om$ and the grid for the simplified
  trapezoid rule needed to satisfy the desired error tolerance. 
  
  In a number of applications, we also used   {\em fractional-parabolic 
  deformations}
 \begin{equation}\label{fractparabasic}
\chi^\pm_{\om;a; b}(\eta)= i\om\pm ib(1\mp i\eta)^a,
\end{equation}
where $\om\in\bR, a>1, b>0$, and {\em hyperbolic} or {\em sub-polynomial  deformations} defined by $\xi=\eta\ln^m(1+b\eta^2)$, where 
$m\ge 1, b>0$. See \cite{ConfAccelerationStable} for the general discussion, applications to efficient evaluation of special functions, analogs of the three families for evaluation of stable distributions with explanation when the seemingly less efficient families are more efficient than the sinh-acceleration,  and further references. In \cite{Sinh},
the composition of fractional-parabolic and sinh changes of variables was used.
 
 The complexity of the numerical scheme 
 based on the sinh-acceleration \eq{sinhbasic} is of the order of $E\ln E$, where $E=\ln(1/\eps)$; in the case of VGP with $\mu=0$, of the order of $O(E^2)$. See \cite{SINHregular} for details. The idea of conformal acceleration method is similar to the idea of the saddle point method.
 However,  the universal families of conformal deformations are simpler to use, especially when deformations of several lines of integration
 are needed, and the deformations must be in a certain agreement 
 \cite{paraLaplace,paired,Contrarian}. 
 
 \subsection{Barrier options} 
 Consider the no-touch barrier option on an asset $S_t=X_t$, with the lower barrier $H=e^h$, and maturity date $T$. The riskless rate $r$ is constant. Assume that under an EMM chosen for pricing, $X$ is SINH-regular of type $((\mum,\mup), \cC, \cC_+)$, where $\mum<0$ and $\bC_+$ contains $\bR\setminus \{0\}$. In \cite{KoBoL,barrier-RLPE,NG-MBS}, we have derived general formulas for prices of barrier options. In the case of the no-touch option, the general formula reduces to
 \bbe\label{VftdEpv3}
V(h;T;x)=\frac{e^{-rT}}{(2\pi)^2i}\int_{\Re q=\sg}e^{qT}q^{-1}\int_{\Im\xi=\om^0}e^{i(x-h)\xi}\frac{\phimq(\xi)}{-i\xi}d\xi\,dq,
\ee
where $\om^0\in (\mum,0)$ is sufficiently small in absolute value, and $\phimq(\xi)$ is the `{\em minus-Wiener-Hopf factor}. Recall that if $q>0$, then $\phimq(\xi)=\bE[e^{i\xi \uX_{T_q}}]$, where $T_q$ is the exponentially distributed random variable of mean $1/q$
independent of $X$, and $\uX_t=\inf_{0\le s\le t}X_s$ is the infimum process. Since $\psi$ admits analytic continuation to the strip
$S_{(\mum,\mup)}$, there exist $\sg_-(q)<0<\sg_+(q)$ s.t. for any $\om'\in (\sg_-(q), \sg_+(q))$  and  any  $\xi\in \{\Im\xi<\om'\}$,
\begin{equation}\label{phim1}
\phimq(\xi)=\exp\left[-\frac{1}{2\pi i}\int_{\Im\eta=\om'}\frac{\xi \ln (q+\psi(\eta))}{\eta(\xi-\eta)}d\eta\right].
\end{equation}
Since $X$ is SINH-regular, the RHS admits analytic continuation w.r.t. $q$ and $\xi$ to unions of strips and coni,
denote these unions $U_1$ and $U_2$, respectively, and the choice of $U_2$ imposes restrictions on the choice of $U_2$.

If the Gaver-Stehfest method is used, one needs to calculate the inner integral in \eq{VftdEpv3} for positive $q$'s only. 
Since $x-h>0$, we deform the inner contour into a contour of the form $\cL_1:=\chi_{\om_1;b;\om}(\bR)$, where $\om\in (0, \gap)$;
the restrictions on the choice of other parameters depend on $q$ and $\psi^0$:   in the process of deformation, we must have 
$q+\psi(\xi)\neq 0$ because the analytic continuation of $\phimq(\xi)$ to $\cC_+\cup \{\Im\xi>0\}$ is defined using the Wiener-Hopf factorization formula
\bbe\label{whf}
\phipq(\xi)\phipq(\xi)=\frac{q}{q+\psi(\xi)},
\ee
where $\phipq(\xi)=\bE[e^{i\xi \barX_{T_q}}]$, and $\barX_t$ is the supremum process. It follows that if $X$ is an SL-process, then
any $\om\in (0,\pi/2)$ can be chosen, and the only restriction is that $q+\chi_{\om_1;b;\om}(0)>0$. If one   of the zeros
of $q+\psi(\xi)$ on $i\bR$ is crossed, the residue theorem is applied. The contour in \eq{phim1} is deformed downward into a contour
$\cL_2:=\chi_{\om'_1;b';\om'}(\bR)$, where $\om'<0$; if $X$ is an SL process, any $\om'\in (-\pi/2,0)$ can be used. The two crucial restrictions are weaker: 1) $\cL_2$ is above $i\mum$; 2) $\cL_1$ is above $\cL_2$ so that the integrand is well-defined in the process of deformation, for all $\xi$.

If the conformal change of variables $q=\tilde\chi_{\sg;b;\om}(y_1)=\sg+ib\sinh(i\om+y_1)$, $\sg>0$, $b>0$, $\om>0$ in the Bromwich integral is used,
then the parameters of $\cL_1,\cL_2$ and $\cL_3:=\tilde\chi_{\sg;b;\om}(R)$ must be chosen so that, in the process of all three deformations, 
$q+\psi(\eta)\neq 0$ and $q+\psi(\xi)\neq 0$, $\xi-\eta\neq 0$. For details, see \cite{paraLaplace}, where the less efficient family of fractional-parabolic deformations was used. The reformulation of the procedures in \cite{paraLaplace} to the case of sinh-deformations is
straightforward.


 If the payoff of the barrier option depends on $X_T$, then the explicit formula
derived in \cite{KoBoL,barrier-RLPE,NG-MBS} involves the triple integral (and two Wiener-Hopf factors must be calculated);
if the option pays when the barrier is crossed, then quadruple integrals appear \cite{Contrarian}. For explicit efficient procedures
for evaluation of these integrals, see \cite{Contrarian}.

\end{document}